\documentclass{article}
\usepackage{amsmath,amssymb,amsthm,mathtools,mathrsfs}
\usepackage{graphicx}
\usepackage[colorlinks=true]{hyperref}

\mathtoolsset{showonlyrefs}
\graphicspath{{fig/}}

\title{Boundary observability for gas giant metrics}
\author{Maarten V. de Hoop\footnote{Rice University, Simons Chair in Computational and Applied Mathematics and Earth Science, TX 77005, USA (\texttt{mvd2@rice.edu})}
\and
Antti Kykk\"anen\footnote{Rice University, Computational applied mathematics and operations research, TX 77005, USA (\texttt{ak272@rice.edu})}
\and
Emmanuel Tr\'elat\footnote{Sorbonne Universit\'e, Universit\'e Paris Cit\'e, CNRS, Inria, Laboratoire Jacques-Louis Lions, LJLL, F-75005 Paris, France (\texttt{emmanuel.trelat@sorbonne-universite.fr}).}}

\date{}

\def\R{\mathrm{I\kern-0.21emR}}
\def\N{\mathrm{I\kern-0.21emN}}
\def\Z{\mathbb{Z}}
\newcommand{\C}{\mathbb{C}}

\newcommand{\energy}{\mathcal{E}}
\newcommand{\abs}[1]{\vert #1 \vert}
\newcommand{\norm}[1]{\Vert #1 \Vert}

\newtheorem{theorem}{Theorem}
\newtheorem{proposition}[theorem]{Proposition}
\newtheorem{corollary}[theorem]{Corollary}

\newtheorem{lemma}[theorem]{Lemma}

\theoremstyle{definition}
\newtheorem{remark}[theorem]{Remark}

\begin{document}
	
\maketitle

\begin{abstract}
We study the observability of waves on gas giant manifolds which are a class of Riemannian manifolds whose metrics are singular at the boundary. Such manifolds arise naturally in modeling of acoustic wave propagation in gas giant planets.

We establish an observability inequality using full boundary measurements given by a Neumann-type trace that is natural in the gas giant setting. The proof proceeds in two steps. First, observability for a general gas giant metric is reduced to the so-called separable case via a perturbation argument. In the separable case, we employ a uniform-in-tangential-frequency analysis combined with an Ingham inequality to prove observability.
\end{abstract}

\section{Introduction}

Consider an~$n+1$ dimensional compact smooth manifold~$X$ with smooth boundary. We equip~$X$ with the singular Riemannian metric
\begin{equation}
\label{eqn:gas-giant-metric}
g
=
\frac{\bar g}{\rho^\alpha}
\end{equation}
where~$\bar g$ is a smooth Riemannian metric on~$X$,~$\rho$ is a boundary defining function and~$\alpha \in (0,2)$ is a parameter. There is a choice of smooth coordinates for~$X$ where the part of~$X$ close to the boundary can be identified with~$[0,1) \times M$ with~$M$ being a closed manifold of dimension~$n$. In these coordinates near~$\partial X$, the Riemannian metric~$g$ can be written in the form
\begin{equation}
g = dx^2 + \frac{g_1(x)}{x^\beta}
\end{equation}
where~$x$ is the coordinate transversal to the boundary,~$\beta > 0$ is a positive parameter and~$g_1(x)$ is a family of smooth Riemannian metrics depending on~$x \in [0,1)$. 
Throughout the paper, we work in such a coordinate system near the singular boundary and identify~$X = [0,1) \times M$.

Singular Riemannian metrics of the form~\eqref{eqn:gas-giant-metric} have been proposed as models for describing wave propagation in gas giant planets (see~\cite{dHIKM2024} for details). Hence such metrics are sometimes referred to as gas giant metrics and the manifold~$(X,g)$ is referred to as a gas giant manifold.

\subsection{Main results}

The purpose of this article is to study observability of acoustic waves on~$X$. In the physical context of gas giant planets the objective is to reconstruct features of the internal structure of the planet based on observations of waves on the surface.
In particular, we are interested in boundary observability, i.e., when wave observations are allowed to be made in a subset~$\omega \subset M$ of the boundary.
Mathematically, observability is modeled by an observability inequality.

It is well-understood and indeed well-documented (see e.g.~\cite{BLR1992}) that validity of an observability inequality is closely tied to the so called geometric control condition, GCC for short: The set~$\omega \subset M$ satisfies GCC if any geodesic ray propagating in~$X$ and reflecting of~$\partial X$ according to the laws of classical optics reaches~$\omega$ in a given finite time. 

In general, GCC fails on gas giants. When $X$ is a closed ball in $\R^{n+1}$, the geometric control condition is never satisfied unless $\omega$ contains the whole boundary. For this reason, our main observability result is stated for the full boundary $\omega = M$. We return to localized observations in Section~\ref{sec:conclusions}. There we first prove a band-limited localized observability estimate and then explain how the abstract moving-observer mechanism developed in the separate paper~\cite{dHKT2026} specializes to the present gas-giant setting, yielding a deterministic Ces\`aro asymptotic observability inequality for time-varying localized observations.

To make matters more concrete, consider the wave equation
\begin{equation}
\label{eqn:wave-equation}
\partial_t^2 u + \triangle_g u = 0
\quad\text{in}\quad (0,T) \times X
\quad\text{with}\quad
u|_{t = 0} = u_0,
\quad
\partial_t u|_{t = 0} = u_1.
\end{equation}
Here~$\triangle_g$ is the Laplace-Beltrami operator of~$(X,g)$. We impose Friedrichs boundary conditions at~$x = 0$ (see Section~\ref{sec:friedrichs} and thereafter) and either Neumann or Dirichlet boundary conditions at~$x = 1$ (the choice of boundary values at $x = 1$ is immaterial, see Remark~\ref{rem:dirichlet-neumann}).

As our main result, we prove an observability inequality in Theorem~\ref{thm:obs-inequality}. The constants in the claim depend on the parameter~$\beta$ and are defined by
\begin{equation}
\label{eqn:geometric-constants}
\nu = \frac{1}{2} + \frac{\beta n}{4}
\qquad\text{and}\qquad
\kappa = \frac{2}{\beta+2}.
\end{equation}

\begin{theorem}
\label{thm:obs-inequality}
Let $u$ solve \eqref{eqn:wave-equation} with the Friedrichs boundary condition at $x = 0$ and either the Dirichlet or the Neumann boundary condition at $x = 1$. Set $\nu$ and $\kappa$ as in \eqref{eqn:geometric-constants}. Then, for every $T > T_\ast$, there exist constants $0 < c_T \leq C_T$ such that
\begin{equation}
\label{main_obs}
\boxed{
 c_T \mathcal{E}_\nu[u_0,u_1]
 \leq  \int_0^T \int_M \Big\vert \lim_{x\to 0^+} x^{\frac{1}{2}-\nu} \partial_x u(x,y,t) \Big\vert^2 dv_G(y) dt
 \leq  C_T \mathcal{E}_\nu[u_0,u_1]
}
\end{equation}
where the anisotropic boundary energy is defined by
\begin{equation}
\label{energies}
\boxed{
 \mathcal{E}_\nu[u_0,u_1]  =  \Vert u_0 \Vert_{H_x^{\nu+\frac{1}{2}} L^2(M)}^2  +  \Vert u_1 \Vert_{H_x^{\nu-\frac{1}{2}} L^2(M)}^2
 +  \Vert \triangle_G^{1/2} u_0 \Vert_{H_x^{\nu-\frac{1}{2}} L^2(M)}^2
}
\end{equation}
Moreover, the threshold $T_\ast = \beta + 2$ is optimal.
\end{theorem}

For a fully polytropic gas giant (see~\cite[Section 1.2]{dHIKM2024}), $\beta = 2$, hence $\kappa = \frac{1}{2}$, $T_\ast = 4$, and $\nu = \frac{1}{2} + \frac{n}{2}$.

\begin{remark}
\label{rem:dirichlet-neumann}
The boundary condition at $x = 1$ does not affect the threshold $T_\ast$ nor the structure of the energy. Indeed, all arguments developed in the proof of Theorem~\ref{thm:obs-inequality} are insensitive to the choice of Dirichlet versus Neumann at $x = 1$: both choices lead to self-adjoint realizations with discrete spectra and the same Weyl constant $\kappa / \pi$.
\end{remark}

In Theorem~\ref{thm:obs-inequality}, $dv_G$ denotes the volume measure of $(M,G = g_1(0))$ and $\triangle_G$ the corresponding Laplace-Beltrami operator. Section~\ref{sec2} contains the definition of the anisotropic spaces appearing in \eqref{energies} as well as the Neumann-like observation trace in \eqref{main_obs}.

A gas giant metric is called separable if $g_1(x) = G$ is independent of $x$. In that case, $\triangle_g$ admits a simple normal form near $x = 0$, which is derived in Section~\ref{sec:normal-form-for-laplacian}. The proof of Theorem~\ref{thm:obs-inequality} first treats the separable case and then deduces the general case by perturbation; both parts are gathered in Section~\ref{sec:boundary-obs-inequality}.

Throughout the article, we use the following notation: for nonnegative quantities $A$ and $B$, we write $A \asymp_T B$ if there exist constants $0 < c_T \leq C_T$ such that $c_T B \leq A \leq C_T B$. Likewise, $A \lesssim_T B$ means $A \leq C_T B$, and $A \gtrsim_T B$ means $A \geq c_T B$.

\subsection{Related results}
Observability and controllability for wave equations has been extensively studied over the past decade. In the standard uniformly hyperbolic setting, observability inequalities are known to hold under certain geometric conditions on the domain, most notably the geometric control condition due to~\cite{BLR1992}. We refer the readers to~\cite{LasieckaTriggiani2000,Lions1988,Miller2005} for detailed expositions on controllability and observability in standard uniformly elliptic settings.

A recent line of research has focused on observability and controllability properties of control systems whose spatial operator degenerates at the boundary of the domain. For wave-type operators, this direction was initiated in~\cite{Gueye2014}, which inspired our analysis in Section~\ref{sec:1D-case}. To the best of our knowledge,~\cite{Gueye2014} remains the only work in this line of research (prior to our contributions) in which the observation or control is applied at the degenerate part of the boundary. Other related contributions  include~\cite{ABCL2017,HB2018,Ouzahra2019,ZG2017}, where degenerate wave equations in non-divergence form are studied in settings closely related to that of Section~\ref{sec:1D-case}. The main difference between our results and the mentioned work is that we make observations at the degenerate part of the boundary as opposed to the non-degenerate part.

Singular Riemannian metrics of the form~\eqref{eqn:gas-giant-metric} have been studied as geometric models for wave propagation on gas giant planets since~\cite{dHIKM2024} where the authors initiated the study of the basic geometric properties of these manifolds and basic analytic properties of the corresponding Laplace-Beltrami operator. The article~\cite{dHIKM2024} also studied certain geometric inverse problems. Extended studies of gas giant geometry have then emerged. In~\cite{CdVDdHT2024}, the authors studied further properties of the Laplace-Beltrami operator and prove Weyl asymptotic laws for the eigenvalues of the operator. A series of works~\cite{Dietze2025,DR2024} studied the concentration of eigenfunctions.

The main contribution of the present work is establishing an observability inequality for the degenerate wave equation on gas giant manifolds.

\subsection{Structure of the article}
In Section~\ref{sec:1D-case}, we study a degenerate $1$-dimensional wave equation and prove the corresponding observability inequality. This model case introduces the one-dimensional spectral analysis and the anisotropic Sobolev scale that reappears in the multidimensional problem.

Section~\ref{sec2} turns to gas giant manifolds. We recall the geometric setting, introduce the conjugation leading to the normal form for $\triangle_g$, analyze the Friedrichs boundary condition at $x = 0$, and define the Neumann-like observation trace used throughout the paper.

The proof of Theorem~\ref{thm:obs-inequality} is split into the separable and the non-separable cases. The separable case is treated by a uniform-in-tangential-frequency argument combined with the one-dimensional analysis. The non-separable case is obtained by a perturbation argument based on Fattorini-Hautus type stability estimates.

Section~\ref{sec:controllability} records the dual exact controllability consequence. Section~\ref{sec:localized-obs-band-limited} explains why fixed localized observations fail in general and derives a positive result under a tangential spectral cutoff. 
Finally, Section~\ref{sec:local-obs-time-varying} records the gas-giant specialization of the abstract moving-observer mechanism developed in the separate paper~\cite{dHKT2026}: exact convexification on finite-dimensional boundary profile spaces, switching realization on each observation interval, and the large-time Ces\`aro tail-reduction argument.

\paragraph{Acknowledgment.}

MVdH was supported by the National Science Foundation under grant DMS-2407456, the Simons Foundation under the MATH + X program and the corporate members of the Geo-Mathematical Imaging Group at Rice University.
ET acknowledges the support of the grant ANR-23-CE40-0010-02 (Einstein-PPF).

\section{A degenerate $1$-dimensional wave equation}
\label{sec:1D-case}

We start by considering a degenerate wave equation with a one spatial dimension. The spatial operator is the same as the leading order degenerate part of the Laplace-Beltrami operator of a gas giant metric. In this section, we establish an observability inequality for this degenerate system (see Theorem~\ref{thm:1d-observability}).

\paragraph{Setting.}

In this section, we study the wave equation
\begin{equation}
\label{eqn:1d-wave-equation}
\partial_t^2u - x^\alpha\partial_x^2u = 0
\quad
\text{in}\quad (0,1) \times (0,T),
\quad
u|_{t=0} = u_0,
\quad
\partial_tu|_{t=0} = u_1.
\end{equation}
with zero Dirichlet boundary conditions at $x = 0$ and $x = 1$.

Some notation is needed to state the main observability result for \eqref{eqn:1d-wave-equation}. We denote by $L^2_\alpha(0,1)$ the Hilbert space of measurable functions $f : (0,1) \to \R$ such that
\begin{equation}
\Vert f \Vert_{L^2_\alpha(0,1)}^2 = \int_0^1 \vert f(x) \vert^2 x^{-\alpha} dx < +\infty.
\end{equation}
The operator $-x^\alpha \partial_x^2$ is self-adjoint and positive on its natural domain in $L^2_\alpha(0,1)$. There exists an orthonormal basis of eigenfunctions $(\Phi_k)_{k \geq 1}$ with eigenvalues $(\lambda_k)_{k \geq 1}$ satisfying
\begin{equation}
-x^\alpha \Phi_k'' = \lambda_k \Phi_k \qquad \text{in } (0,1), 
\qquad \Phi_k(0) = \Phi_k(1) = 0.
\end{equation}
For $s \geq 0$, we define
\begin{equation}
H^s(0,1) = \{  u \in L^2_\alpha(0,1)  \,:  \norm{u}_s < +\infty \}
\end{equation}
where, if $u = \sum_{k \geq 1} a_k \Phi_k$, then
\begin{equation}
\norm{u}_s^2 = \sum_{k \geq 1} \vert a_k \vert^2 \lambda_k^s.
\end{equation}

\begin{theorem}
\label{thm:1d-observability}
Let $\alpha \in (0,2)$ and define $\nu = 1/(2-\alpha)$, $\kappa = 1-\alpha/2$. Let $u$ solve~\eqref{eqn:1d-wave-equation}. Then, for every $T > T^* = 2/\kappa$, there exist constants $0 < c_T \leq C_T$ such that
\begin{equation}
c_T\, \energy_\nu[u_0,u_1] \leq \int_0^T \abs{\partial_xu(0,t)}^2 \,dt \leq C_T\, \energy_\nu[u_0,u_1]
\end{equation}
where the boundary energy is
\begin{equation}
\energy_\nu[u_0,u_1] = \norm{u_0}_{H^{\nu+\frac{1}{2}}}^2 + \norm{u_1}_{H^{\nu-\frac{1}{2}}}^2.
\end{equation}
Moreover, the threshold $T^* = 2/\kappa$ is optimal.
\end{theorem}

The proof of the theorem is based on an explicit computation of the eigenfunctions and uses several standard properties of Bessel functions.
The explicit computation of $\Phi_k$ is carried out in the next section.

\subsection{Proof of Theorem~\ref{thm:1d-observability}}

We start by explicitly computing the eigenvalues and eigenfunctions of the operator $-x^\alpha\partial_x^2$. Hence we will solve the equation
\begin{equation}
\label{eqn:alpha-bessel}
-x^\alpha y''(x) = \lambda y(x)
\quad
\text{in}
\quad
(0,1)
\quad
\text{with}
\quad
y(0) = y(1) = 0.
\end{equation}
Let us set
\begin{equation*}
\nu = \frac{1}{2-\alpha}
\qquad \text{and} \qquad
\kappa = 1 - \frac{\alpha}{2}.
\end{equation*}

\begin{proposition}
\label{prop:bessel-functions}
Let $\alpha \in (0,2)$. The eigenvalues in~\eqref{eqn:alpha-bessel} are $\lambda_k = (\kappa j_{\nu,k})^2 $ for $k \in \N$ and the corresponding solutions to~\eqref{eqn:alpha-bessel} are $y_k(x) = x^{\frac12} J_{\nu}(j_{\nu,k}x^\kappa)$ where $J_\nu$ is the $\nu$-th Bessel function of the first kind and $j_{\nu,k}$ is the $k$-th zero of $J_{\nu}$. The $L^2_\alpha(0,1)$ norm of $y_k$ is equal to $(2\kappa)^{-1/2}\abs{J'_\nu(j_{\nu,k})}$.
 \end{proposition}
 
 \begin{proof}
 Equation~\eqref{eqn:alpha-bessel} is a special case of the more general Bessel equation
 \begin{equation}
 \label{eqn:bessel-equation}
 x^2y''(x) + axy'(x) + (bx^l + c)y(x) = 0.
 \end{equation}
 When $b \neq 0$, the solutions to~\eqref{eqn:bessel-equation} are of the form $y(x)  =  x^{\frac12(1-a)} Z_\nu(\kappa^{-1}\sqrt{b}x^\kappa)$ where $Z_\nu$ is any Bessel function and the constant $\nu$ and $\kappa$ are determined by the coefficients as
$$
\nu = \frac{1}{l}\sqrt{(1-a)^2 - 4c}  \quad  \text{and}  \quad  \kappa = \frac{l}{2}.
$$ 
For the special case of interest, equation~\eqref{eqn:alpha-bessel} the coefficients are~$a = 0$,~$b = \lambda$,~$c = 0$ and~$l = 2-\alpha$. Hence $\nu = \frac{1}{2-\alpha}$ and $\kappa = 1 - \frac{\alpha}{2}$ as defined before the proposition.

Introducing the boundary values as in~\eqref{eqn:alpha-bessel} forces us to take~$Z_\nu $ to be the Bessel function of the first kind,~$J_\nu$, since we must have~$y(0) = 0$. Then the other boundary condition provides us that~$J_\nu(\kappa^{-1}\sqrt{\lambda}) = 0$.  The Bessel function has countably many zeros~$j_{\nu,k}$ for $k \in \N$ and we must have ~$j_{\nu,k} = \kappa^{-1}\sqrt{\lambda}$. Hence there are countably many possible choices for~$\lambda$ which we will denote by~$\lambda_k = (\kappa j_{\nu,k})^2$. These are the eigenvalues as claimed. We have also shown that the corresponding solutions to~\eqref{eqn:alpha-bessel} are
$y_k(x) = x^{\frac12}J_\nu(j_{\nu,k}x^\kappa)$.

Next, we will compute the norm of $y_k$. To accomplish this, we use an orthogonality property of the Bessel functions: For any $n,m\in\N$ it holds that
$$
\int_0^1 xJ_{\nu}(j_{\nu,n}x)J_{\nu}(j_{\nu,m}x) \,dx = \frac{\delta_{nm}}{2} (J_{\nu+1}(j_{\nu,n}))^2.
$$
Thus, changing variables by $s = x^{1-\frac{\alpha}{2}}$ gives
$$
\int_0^1 y_n(x)y_m(x) \,x^{-\alpha}dx = \kappa^{-1} \int_0^1 sJ_{\nu}(j_{\nu}s)J_{\nu}(j_{\nu,m}s) \,ds = \kappa^{-1} \frac{\delta_{nm}}{2} (J_{\nu+1}(j_{\nu,n}))^2. 
$$
In particular, the norm of $y_k$ is
\begin{equation}
\norm{y_k}_0 = (2\kappa)^{-\frac{1}{2}} \abs{J_{\nu+1}(j_{\nu,k})} = (2\kappa)^{-\frac{1}{2}} \abs{J_{\nu}'(j_{\nu,k})}.
\end{equation}
The last step above is valid due to the recurrence rule of the Bessel functions:
\begin{equation}
\label{eqn:bessel-recurrence}
xJ_\nu'(x) - \nu J_\nu(x) = -xJ_{\nu+1}(x).
\end{equation}
This completes the proof.
 \end{proof}
 
Proposition~\ref{prop:bessel-functions} shows that the eigenfunction basis of the operator $-x^\alpha\partial_x^2$ is formed by the functions $ \Phi_k(x) = (2\kappa)^{\frac{1}{2}}\abs{J'_{\nu}(j_{\nu,k})}^{-1}y_k(x)$.

 \begin{proof}[Proof of Theorem~\ref{thm:1d-observability}]
We represent the initial data in equation~\eqref{eqn:1d-wave-equation} using the eigenfunction from Proposition~\ref{prop:bessel-functions} as
$$
u_0(x) = \sum_{k \geq 0} u_0^k\Phi_k(x) \qquad \text{and} \qquad u_1(x) = \sum_{k \geq 0} u_1^k\Phi_k(x).
$$
Let $\omega_k = \sqrt{\lambda_k} = \kappa j_{\nu,k}$ for any $k \in \N$. Then expanding the solution $u$ in a Fourier series in $t$ and using Proposition~\ref{prop:bessel-functions} we find that
$$
u(x,t) = \sum_{k \in \N} u_k(t)\Phi_k(x) \qquad \text{where} \qquad u_k(t) = a_k e^{i\omega_kt} + a_{-k}e^{-i\omega_kt}.
$$
The coefficients $a_{\pm k}$ are determined by the initial values; we must have $a_k + a_{-k} = u_0^k$ and $i\omega_k(a_k - a_{-k}) = u_1^k$, or equivalently,
$$
a_k = \frac12 \left( u_0^k - i\frac{u_1^k}{\omega_k} \right) \qquad \text{and} \qquad a_{-k} = \frac12 \left( u_0^k
+ i\frac{u_1^k}{\omega_k} \right).
$$
The next step is to compute the Neumann trace of $u$ as $x\to 0$. We let $C_{\nu,k} = (2\kappa)^{\frac{1}{2}}\abs{J'_{\nu}(j_{\nu,k})}^{-1}$. Then differentiating $\Phi_k$ gives
\begin{equation}
\label{eqn:bessel-derivative}
\Phi_{k}'(x) = C_{\nu,k} \left( \frac{1}{2x^{\frac12}} J_\nu(j_{\nu,k}x^\kappa) + j_{\nu,k}x^{\kappa - \frac{1}{2}} J_\nu'(j_{\nu,k}x^\kappa) \right) 
\end{equation}
We use the recurrence property~\eqref{eqn:bessel-recurrence} of Bessel functions to write
$$
j_{\nu,k}x^{\kappa - \frac{1}{2}} J_\nu'(j_{\nu,k}x^\kappa) = \kappa\nu x^{-\frac{1}{2}}J_\nu(j_{\nu,k}x^\kappa) - \kappa j_{\nu,k}x^{\kappa - \frac12} J_{\nu +1}(j_{\nu,k}x^\kappa).
$$
Given that $\kappa\nu=\frac12$ the derivative in~\eqref{eqn:bessel-derivative} becomes
$$
\Phi'_{k}(x) = C_{\nu,k} \left( \frac{1}{2x^{\frac12}} J_\nu(j_{\nu,k}x^\kappa) - j_{\nu,k}x^{\kappa - \frac{1}{2}} J_{\nu+1}(j_{\nu,k}x^\kappa) \right)
$$
Then we use the well-known series expansions for the Bessel functions $J_\nu$ abd $J_{\nu+1}$ to write
$$
J_\nu(j_{\nu,k}x^\kappa) = \sum_{n \geq 0} \frac{(-1)^n}{n!\Gamma(n+\nu+1)} \left( \frac{j_{\nu,k}x^\kappa}{2} \right)^{2n+\nu} = \frac{(j_{\nu,k})^\nu}{2^\nu\Gamma(\nu+1)}x^{\frac{1}{2}} + \mathrm{O}(x^{\frac{1}{2}+2\kappa}) 
$$
and
$$
J_{\nu+1}(j_{\nu,k}x^\kappa) = \sum_{n \geq 0} \frac{(-1)^n}{n!\Gamma(n+\nu+2)} \left( \frac{j_{\nu,k}x^\kappa}{2} \right)^{2n+\nu+1} = \frac{(j_{\nu,k})^{\nu+1}}{2^{\nu+1}\Gamma(\nu+2)}x^{\frac{1}{2}} + \mathrm{O}(x^{\frac{1}{2}+3\kappa}).
$$
Remembering that $C_{\nu,k} = (2\kappa)^{\frac12}\abs{J'_\nu(j_{\nu,k})}^{-1}$ the above gives
$$
\Phi'_{k}(x) = C_{\nu,k} \frac{(j_{\nu,k})^\nu}{2^\nu\Gamma(\nu+1)} + \mathrm{O}(x^{2\kappa}) = \tilde C_{\nu}
\frac{(j_{\nu,k})^\nu}{\abs{J'_{\nu}(j_{\nu,k})}} + \mathrm{O}(x^{2\kappa}) 
$$
for a new constant $\tilde C_\nu$ independent of $k$.

Inserting the asymptotics of the derivative $\Phi'_k(x)$ into the expansion of the solution $u$ gives
$$
\lim_{x\to 0^+} \partial_xu(x,t) \asymp_\alpha \sum_{k \geq 0} \frac{(j_{\nu,k})^\kappa}{\abs{J'_\nu(j_{\nu,k})}}
(b_ne^{i\omega_kt}+b_{-n}e^{-i\omega_kt}) = \sum_{k \in \Z} a_ke^{i\tilde\omega_kt}
$$
where $\tilde\omega_n = \omega_n$, $\tilde\omega_{-n} = -\omega_{n}$ and $a_n = (j_{\nu,n})^\nu\abs{J'_{\nu}(j_{\nu,n})}^{-1}b_n$, $a_{-n} = (j_{\nu,n})^\nu\abs{J'_{\nu}(j_{\nu,n})}^{-1}b_{-n}$ for $n \in \N$.
The sequence $(\tilde\omega_n)$ satisfies the gap condition since the sequence $(j_{\nu,n+1}-j_{\nu,n})_n$ is strictly decreasing due to $\abs{\nu} > 1/2$ which uses the fact that $\alpha \in (0,2)$.

The sequence $(\tilde\omega_n)$ satisfies the gap condition since the sequence $(j_{\nu,n+1}-j_{\nu,n})_n$ is strictly decreasing due to $\abs{\nu} > 1/2$ which uses the fact that $\alpha \in (0,2)$. Hence we can use Ingham inequality (see Remark~\ref{rem:ingham} after the proof) to deduce that, when $T > 2\pi/\gamma$ with $\gamma \coloneqq \kappa\pi$, we have
$$
\int_0^T \abs{\partial_xu(x,t)}^2|_{x=0}\,dt = \int_0^T 	\abs{\sum_{k \in \Z}a_ke^{i\tilde\omega_kt}}^2 \,dt
\asymp_\alpha \sum_{k \in \Z} \abs{a_k}^2 
\asymp \sum_{k \in \N} \frac{(j_{\nu,n})^{2\nu}}{[J'_\nu(j_{\nu,n})^2]} \left( \abs{u^n_0}^2 + \frac{\abs{u^n_1}^2}{\abs{\omega_n}^2} \right)
$$
To manipulate this further we use the asymptotics
$$
J'_{\nu}(j_{\nu,k}) \asymp \sqrt{\frac{2}{\pi j_{\nu,k}}} \asymp \frac{1}{\sqrt{j_{\nu,k}}}.
$$
We also make use of the fact that $\omega_n = \sqrt{\lambda_n}$ and introduce $\kappa$. By doing so we obtain
$$
\int_0^T (\partial_xu(x,t))^2|_{x=0} \,dt \asymp_\alpha \sum_{k \in \N} (\kappa j_{\nu,n})^{2\nu+1} \left( \abs{u^n_0}^2 + \frac{\abs{u^n_1}^2}{\lambda_n} \right) 
= \sum_{k \in \N} \left( \lambda_n^{\nu+\frac{1}{2}} \abs{u^n_0}^2 + \lambda_n^{\nu-\frac{1}{2}} \abs{u^n_1}^2 \right).
$$
By the definition of the spaces $H^s_\alpha$ the last quantity is comparable to
$\norm{u_0}^2_{H^{\nu+1/2}_\alpha} + \norm{u_1}_{H^{\nu-1/2}_\alpha}^2$,
which proves the observation inequality.

The optimality of the lower bound on $T$ follows from Remark~\ref{rem:upper-density}.
\end{proof}

\begin{remark}[Beurling upper density]
\label{rem:upper-density}
Let $\Omega=\{\omega_k\}_{k\in\mathbb{Z}}\subset\R$ be a discrete multiset of real frequencies with no finite accumulation points, and let $N(x)=\#\{\omega\in\Omega:\ \omega\leq x\}$ be its counting function. The \emph{Beurling upper density} is
$$
D^+(\Omega)=\limsup_{R\to\infty}\ \sup_{x\in\R}\ \frac{N(x+R)-N(x)}{R}.
$$
It measures the maximal average number of frequencies per unit length in arbitrarily long windows, is invariant under finite perturbations of $\Omega$, and (for uniformly separated sequences) coincides with the reciprocal of the asymptotic gap: if $g=\lim_{|k|\to\infty}(\omega_{k+1}-\omega_k)$ exists, then $D^+=1/g$. In our setting $\omega_{\pm n}=\pm\mu_n$ with $\mu_n\sim\kappa\pi n$, hence $g=\kappa\pi$ and $D^+=1/(\kappa\pi)$. With this normalization, the Ingham–Beurling theorem yields the observability/frame inequality on $[0,T]$ as soon as $T>2\pi D^+$, i.e., $T>2/\kappa$.
\end{remark}

\begin{remark}[Ingham inequality]
\label{rem:ingham}
In full generality, let $\Omega=\{\omega_k\}_{k\in\mathbb{Z}}\subset\R$ be a discrete set of real frequencies with a uniform gap
$\inf_{k\neq m}|\omega_k-\omega_m|\geq\gamma>0$.
Then, by~\cite{Ingham1936}, for every $T>0$ sufficiently large (equivalently, as soon as $T>2\pi D^+(\Omega)$ when one formulates the threshold via Beurling’s upper density) there exist constants $c_T,C_T>0$ such that, for every finitely supported sequence $(a_k)$,
$$
c_T \sum_{k\in\mathbb{Z}} |a_k|^2 \leq \int_0^{T}\Big|\sum_{k\in\mathbb{Z}} a_k e^{i\omega_k t}\Big|^2 dt \leq C_T \sum_{k\in\mathbb{Z}} |a_k|^2.
$$
What we use here is the standard uniform-gap Ingham inequality.
We simply restate the threshold as $T>2\pi D^+$ because in our setting $\omega_{\pm n}=\pm\mu_n$ with $\mu_n\sim \kappa\pi n$, hence the asymptotic gap is $g=\kappa\pi$ and $D^+=1/g=1/(\kappa\pi)$, yielding the same sharp condition $T>2/\kappa$.
We refer the reader to~\cite{BKL2002,Haraux1989,KL2005} for more details.
\end{remark}

\section{Boundary observability for gas giants}\label{sec2}

\paragraph{Setting.}
Let $X=(0,1)\times M$ where $(M,g_1)$ is a smooth compact Riemannian manifold of dimension $n\geq 1$, with the metric
\begin{equation}\label{eq:metric}
	g = dx^2 + x^{-\beta}\,g_1(x), \qquad \beta>0, \qquad g_1\in C^{\infty}\big([0,1);{\rm Met}(M)\big) ,
\end{equation}
where $g_1(x)$ is a smooth $x$-dependent Riemannian metric on the boundary manifold $M$.
Let $dv_G$ denote the Riemannian volume on $M$ and let $\widetilde\triangle_g$ be the Laplace-Beltrami operator on $(X,g)$ with Friedrichs boundary condition at $x=0$ (see Section \ref{sec:friedrichs} further) and either Dirichlet or Neumann condition at $x=1$; we adopt the ``geometric" convention and assume that $\widetilde\triangle_g$ is nonnegative.
In the local coordinates, we have
$$
dv_g = x^{-\frac{\beta n}{2}} \left( \det g_1(x,y) \right)^{1/2}\, dx\, dy = x^{-\frac{\beta n}{2}} dx\, dv_{g_1(x)} 
$$
with
$$
dv_{g_1(x)} = \left( \frac{\det g_1(x,y)}{\det g_1(0,y)} \right)^{\frac{1}{2}} \underset{x\rightarrow 0^+}{\longrightarrow} dv_G
$$
with $G=g_1(0)$. In the \emph{separable} case, we have $g_1(x)=G$.

We investigate the wave equation
$$
\boxed{
\partial_t^2z+\widetilde\triangle_gz=0
}
$$
and seek an observability inequality with the natural Neumann boundary observation at $x=0$, which is the physical flux (coming from the Stokes formula)
\begin{equation}\label{physical_flux}
	\boxed{
		\mathcal{C}_{\textrm{phys}}(z) = \lim_{x\rightarrow 0^+} x^{-\frac{\beta n}{2}}\!\left(\frac{\det g_1(x,y)}{\det g_1(0,y)}\right)^{\frac{1}{2}}\partial_x z = \lim_{x\rightarrow 0^+} x^{-\frac{\beta n}{2}} \partial_x z  \quad\textrm{in}\ L^2(M,dv_G)
	}
\end{equation}

\paragraph{Conjugation.}
We define the unitary operator 
$$
U:L^2(X,dv_g)\rightarrow L^2((0,1)\times M,dx\, dv_G(y))
$$
by 
\begin{equation}\label{eq:unitary-U}
u(x,y) = (Uz)(x,y) = m(x,y)z(x,y),
\end{equation}
with
$$
m(x,y) = x^{-\frac{\beta n}{4}}\rho(x,y)
\qquad\textrm{and}\qquad
\rho(x,y)=\left(\frac{\det g_1(x,y)}{\det g_1(0,y)}\right)^{\frac{1}{4}}
$$
(note that $\rho$ is smooth and $\rho(0,y)=1$), and we define on $L^2((0,1)\times M,dx\, dv_G(y))$ the nonnegative operator
$$
\triangle_g = U \widetilde\triangle_g U^{-1} .
$$
Note that, in~\cite{CdVDdHT2024}, the authors used the conjugation by the density $x^{-\frac{\beta n}{4}}$, but since it is not exactly unitary (this changes some multiplicative constants though, which is harmless). We prefer to use the one above.
Note anyway that, near $x=0$, we have $u \sim x^{-\frac{\beta n}{4}}z$.

From now on, we study the wave equation
\begin{equation}\label{eq:wave}
	\partial_t^2 u + \triangle_g u = 0 \quad \text{on }(0,1)\times M,\qquad
	u|_{t=0}=u_0,\ \ \partial_t u|_{t=0}=u_1,
\end{equation}
with Friedrichs boundary condition at $x=0$ (see Section \ref{sec:friedrichs} hereafter) and either Dirichlet or Neumann condition at $x=1$.
In order to derive the right Neumann boundary observation at $x=0$, corresponding to \eqref{physical_flux} in the new coordinate $u$, we need to analyze in more detail the Friedrichs condition at $x=0$. But we first establish a normal form for the Laplacian, in the next section.

\subsection{Normal form for the Laplacian}
\label{sec:normal-form-for-laplacian}
This section is dedicated to establishing a normal form for $\triangle_g$ near $x=0$.

\begin{lemma}\label{lem_normal_form}
There exist $\delta\in(0,1)$ and $C>0$ such that, on $(0,\delta]\times M$,
\begin{equation}\label{eq:L-normal-form}
\triangle_g = -\partial_x ^2 +\frac{C_{\beta}}{x^2} + x^{\beta}\triangle_G + x^\beta\big(\triangle_{G(x)}-\triangle_G\big) \mathcal{T}(x,y,D_y) + \mathcal{R}(x,y,D_x ,D_y) 
\end{equation}
where
\begin{itemize}
\item $\triangle_G=-\mathrm{div}_G\nabla_y\geq 0$ is the Laplace-Beltrami operator on $(M,G=g_1(0))$;
\item $C_{\beta}=\frac{\beta n}{4}\left(\frac{\beta n}{4}+1\right)=\nu^2-\frac{1}{4}$ with $\nu=\frac{1}{2}+\frac{\beta n}{4}$;
\item $x^\beta\big(\triangle_{G(x)}-\triangle_G\big)$ is a \emph{purely tangential} operator of order $2$ with coefficients in $C^1$ satisfying $\big\| \triangle_{G(x)}-\triangle_G\big\|_{C^1(M)}\leq Cx$;
\item $\mathcal{T}(x,y,D_y)$ is a \emph{tangential} differential operator of order $1$ (it involves no $\partial_x $), satisfying $\|\mathcal{T}(x,\cdot,\cdot)\|_{C^1(M)}\leq Cx$;
\item $\mathcal{R}(x,y,D_x ,D_y)$ is of total order $\leq 2$, contains no purely tangential second-order part
(those are kept in $x^\beta(\triangle_{G(x)}-\triangle_G)$), and collects:
a zero-order potential, mixed terms $\partial_xD_y$, and a possible first-order term in $\partial_x$,
all with $C^1$ coefficients bounded by $Cx$.
\end{itemize}
In particular, $\mathcal{T}(0,\cdot,\cdot)=\mathcal{R}(0,\cdot,\cdot)=0$.\\
In the separable case where $g_1(x)\equiv G$, we have $\mathcal{T}=0$ and $\mathcal{R}=0$, and thus
\begin{equation}
\triangle_g = \triangle_g^{\textrm{sep}} = -\partial_x ^2 +\frac{C_{\beta}}{x^2} + x^{\beta}\triangle_G.
\end{equation}
\end{lemma}

The lemma shows that, in the non-separable case, one has $\triangle_g = \triangle_g^{\textrm{sep}} + O(x)$ in the sense made precise in the above statement.

This is different from what is done in Section 2 of~\cite{CdVDdHT2024} where a quasi-isometry property was used. 
The lemma above gives a more precise form but uses in a crucial way the fact that the metric is smooth.
Anyway, for a use of the quasi-isometry property, see Remark \ref{rem_quasiisom} further.

\begin{proof} We proceed in several steps.
	
	\smallskip
	\emph{Step 1: Divergence form and basic ingredients.}
	In local coordinates $(x,y^1,\ldots,y^n)$, write $G(x,y)=(G_{ij}(x,y))_{1\leq i,j\leq n}$ for the matrix of $g_1(x)$ and $G^{ij}(x,y)$ for its inverse.
	The metric $g$ is block-diagonal with blocks $1$ and $x^{-\beta}G(x,y)$, hence
	$$
	g^{xx}=1,\qquad g^{ij}=x^{\beta}G^{ij}(x,y),\qquad |\det g|^{1/2}=x^{-\frac{\beta n}{2}}\big(\det G(x,y)\big)^{1/2}.
	$$
	Thus, for $f\in C_{c}^{\infty}$,
	\begin{align}
		- \widetilde\triangle_g f &= -\mathrm{div}_g\nabla_g f \nonumber\\
		&= |\det g|^{-1/2}\,\partial_x \Big(|\det g|^{1/2}\,g^{xx}\,\partial_x f\Big)
		+ |\det g|^{-1/2}\,\partial_{y^i}\Big(|\det g|^{1/2}\,g^{ij}\,\partial_{y^j}f\Big) \nonumber\\
		&= \partial_x ^2f
		+ \Big(\partial_x \log |\det g|^{1/2}\Big)\,\partial_x f
		+ x^{\beta}\triangle_{G(x)}f
		+ x^{\beta}\,\mathbf{B}^i(x,y)\,\partial_{y^i}f,
		\label{eq:raw-Delta}
	\end{align}
	where the last term collects the first-order tangential drift coming from $y$–derivatives of the density and of $G^{ij}$, and
	$$
	\partial_x \log |\det g|^{1/2} 
	= -\frac{\beta n}{2x} + \frac{1}{2}\,\mathrm{Tr}\big(G^{-1}\partial_x G\big).
	$$
	Here, $\triangle_{G(x)}=-\mathrm{div}_{G(x)}\nabla_y \geq 0$ is the Laplace-Beltrami operator on $(M,G(x))$.
	By Taylor expansion of $G(x)$ at $x=0$ we have, for some $\delta\in(0,1)$,
	\begin{equation}\label{eq:Taylor-G}
		G(x,y)=G(0,y)+x\,H(y)+\mathrm{O}_{C^{\infty}}(x^2),\qquad 
		G^{ij}(x,y)=G^{ij}(0,y)+\mathrm{O}_{C^{\infty}}(x),
	\end{equation}
	and\footnote{\emph{Notation $\mathrm{O}_{C^{\infty}}(x^m)$.}
		Given any smooth scalar function or tensor field $A(x,y)$ on $[0,\delta)\times M$, the notation $A(x,y)=\mathrm{O}_{C^{\infty}}(x^m)$ as $x\rightarrow 0^+$ if there exists $\delta>0$ such that, for all integers $k,\ell\geq 0$, there exists $C_{k,\ell}>0$ such that
		$$
		\sup_{0<x\leq\delta}\ \sup_{y\in M}\ \big\| \partial_x^{\,k}\nabla_y^{\,\ell} A(x,y)\big\|
		\ \leq\ C_{k,\ell}\, x^{\,m-k}.
		$$
		Here $\nabla_y^{\,\ell}$ denotes any fixed family of tangential differential operators of order $\ell$ with smooth coefficients (e.g. iterated covariant derivatives for a chosen background connection on $M$); the choice is immaterial when $M$ is compact.
		Equivalently, $A(x,y)=x^mB(x,y)$ with $B\in C^{\infty}([0,\delta)\times M)$ and all mixed derivatives $\partial_x^{k}\nabla_y^{\ell}B$ uniformly bounded on $[0,\delta)\times M$.
		In particular, tangential derivatives preserve the power $x^m$, while each $x$–derivative lowers it by one: $\partial_x^{k}A=\mathrm{O}_{C^{\infty}}(x^{m-k})$. 
		For example, $G(x,y)=G(0,y)+x\,H(y)+\mathrm{O}_{C^{\infty}}(x^2)$ means every component (in local charts) satisfies these estimates.}
	\begin{equation}\label{eq:trGprime}
		\mathrm{Tr}\big(G^{-1}\partial_x G\big)(x,y)
		= \mathrm{Tr}\big(G(0,y)^{-1}H(y)\big)+\mathrm{O}_{C^{\infty}}(x).
	\end{equation}
	Moreover, in \eqref{eq:raw-Delta}, we have $\mathbf{B}^i(x,y)=\mathrm{O}_{C^{\infty}}(1)$ and $\mathbf{B}^i(0,y)=0$ because the density factor for the tangential block does not depend on $y$ when $x=0$ (see Step~3 below).
	
	\smallskip
	\emph{Step 2: Unitary conjugation to the product measure.}
	Recall that $U$, defined by \eqref{eq:unitary-U}, is unitary from $L^2(X,dv_g)$ onto $L^2\big((0,1)\times M,dx\,dv_G(y)\big)$. Write $m=e^{-\phi}$ with
	$$
	\phi(x,y) = -\ln m(x,y) = \frac{\beta n}{4}\log x - \frac{1}{4}\log\frac{\det G(x,y)}{\det G(0,y)}.
	$$
	For any smooth $f$,
	\begin{equation}\label{eq:conj}
		\triangle_g f = e^{-\phi}\widetilde\triangle_g\big(e^{\phi}f\big)
		= \widetilde\triangle_gf - 2\langle\nabla\phi,\nabla f\rangle_g + \big(\widetilde\triangle_g\phi - |\nabla\phi|_g^2\big)\,f.
	\end{equation}
	We analyze separately the $x$- and $y$-parts of the terms in \eqref{eq:conj}.

	\smallskip
	\emph{Step 3: Cancellation of the singular first derivative in the normal direction.}
	Using \eqref{eq:raw-Delta}, the $x$-part of $\widetilde\triangle_g$ is $-\partial_x^2 f - b(x,y)\,\partial_x f$ with
    $$
    b=\partial_x\log|\det g|^{1/2} =-\frac{\beta n}{2x}+\frac{1}{2}\mathrm{Tr}\big(G^{-1}\partial_x G\big).
    $$
	Moreover,
	$$
	\partial_x \phi=\frac{\beta n}{4x}-\frac{1}{4}\,\partial_x \log\det G(x,y)
	=\frac{\beta n}{4x}-\frac{1}{4}\,\mathrm{Tr}\big(G^{-1}\partial_x G\big) = -\frac{1}{2}b,
	$$
	thus $b+2\,\partial_x\phi=0$.
	In the conjugation identity \eqref{eq:conj}, the cross term
    $$
    -2\langle\nabla\phi,\nabla f\rangle_g=-(b+2\partial_x\phi)\partial_x f-2\langle\nabla_y\phi,\nabla_y f\rangle_g
    $$
    cancels the entire first-order normal term (the tangential piece is absorbed into $\mathcal{T}(x,y,D_y)$).
	Therefore, the normal part of $\triangle_g$ has no first-order term and can be written as $-\partial_x^2+V_x$ where
    $$
    V_x = -\partial_x^2\phi - \big(\partial_x\phi\big)^2 - b\,\partial_x\phi =\frac{1}{2}\,\partial_x b+\frac{1}{4}\,b^2
    $$
    because $b+2\partial_x\phi=0$, 
	which makes the cancellation of any $x^{-1}\partial_x\log\det G$ term immediate. 
	After a harmless smooth redefinition of $x$ preserving the form $dx^2+x^{-\beta}G(x)$, one may assume that $\partial_x\log\det G(0,y)=0$, so that $\partial_x\log\det G(x,y)=\mathrm{O}(x)$.
	Expanding near $x=0$ yields
    $$
    V_x = \frac{C_{\beta}}{x^2}+\mathrm{O}(1)
    $$
	where the $O(1)$ term is smooth up to $x=0$ (it comes from $\partial_x \big[\mathrm{Tr}(G^{-1}\partial_x G)\big]$ and from $(\partial_x \log\det G)^2$, both bounded by \eqref{eq:Taylor-G}--\eqref{eq:trGprime}).
	
	\smallskip
	\emph{Step 4: Tangential second-order part.}
	The tangential second-order part of $\widetilde\triangle_g$ is $x^{\beta}\triangle_{G(x)}$. 
	Expanding $G(x)$ at $0$ as in \eqref{eq:Taylor-G} yields $\big\|\triangle_{G(x)}-\triangle_G\big\|_{C^1(M)}=\mathrm{O}(x)$ and we write
    $$
    x^{\beta}\triangle_{G(x)} = x^{\beta}\triangle_G + x^{\beta}\big(\triangle_{G(x)}-\triangle_G\big).
    $$
    We keep the purely tangential second-order correction in the explicit block $x^\beta(\triangle_{G(x)}-\triangle_G)$, as stated in the lemma; it has $C^1$-coefficients $O(x)$. 
	
	\smallskip
	\emph{Step 5: First-order tangential terms and mixed terms.}
	From \eqref{eq:conj} we get the additional drift $2\langle\nabla\phi,\nabla f\rangle_g$.
	Its tangential part is 
	$$
	2\,\langle\nabla_y\phi,\nabla_yf\rangle_g = 2\,x^{\beta}\,\langle \nabla^{G(x)}\phi,\nabla^{G(x)}f\rangle,
	$$
	with
	$$
	\nabla_y\phi =\frac{1}{4}\,\nabla_y\log\det G(x,y)-\frac{1}{4}\,\nabla_y\log\det G(0,y).
	$$
	At $x=0$ this vanishes identically, and by Taylor expansion it is a $O(x)$ in $C^1(M)$. Hence this yields a first-order tangential operator $\mathcal{T}(x,y,D_y)$ with coefficients $O(x)$ in $C^1(M)$.
	
	Next, the $y$–dependence of $G(x)$ in the $x$–divergence part of $\triangle_g$ also contributes first-order tangential drift terms. They are precisely the $\mathbf B^i$ terms in \eqref{eq:raw-Delta}, which vanish at $x=0$ (because $|\det g|^{1/2}g^{ij}$ then depends only on $y$ and the divergence is taken with respect to $dv_{G(0)}$) and are $O(x)$ by \eqref{eq:Taylor-G}. 
	Finally, the $y$–dependence of $\phi$ induces mixed terms $\partial_x \partial_y$ when expanding $L=e^{-\phi}\triangle_ge^{\phi}$; their coefficients are also $O(x)$ because $\partial_y\phi=\mathrm{O}(x)$ and the normal coefficients of $\triangle_g$ are smooth up to $x=0$.
	
	\smallskip
	\emph{Step 6: Collecting terms and $C^1$–bounds.}
	Steps 3, 4 and 5 yield \eqref{eq:L-normal-form}.
	The $C^1(M)$-bound follows from the expansions \eqref{eq:Taylor-G}--\eqref{eq:trGprime} and the fact that each coefficient of $\mathcal{T}$ and $\mathcal{R}$ contains at least one factor vanishing at $x=0$ (either $\partial_y\phi$, the difference $G(x)-G(0)$, or a derivative thereof), hence is a $O(x)$ as well as its $y$-derivatives.
	The lemma is proved. 
\end{proof}

\paragraph{Separable case.}
In the separable case, we have
\begin{equation}\label{eq:separated}
	\triangle_g = \triangle_g^{\textrm{sep}} = -\partial_x^2 + \frac{C_\beta}{x^2} + x^{\beta}\triangle_G
\end{equation}
where $\triangle_G$ is the Laplace-Beltrami operator on $(M,G)$, of eigenvalues $0\leq\omega_1\leq\cdots\leq\omega_j\leq\cdots$ with an orthormal eigenbasis $(\psi_j)_{j\in\N^*}$, i.e., $\triangle_G \psi_j = \omega_j \psi_j$.
After expanding a function $u$ as $u(x,y)=\sum_{j\geq 1} u_j(x)\,\psi_j(y)$ one has the direct sum decomposition
$$
\triangle_g =  \bigoplus_{k\geq 1} P_{\omega_k}
$$
where
\begin{equation}\label{eq:Pomega-app}
	P_\omega = -\partial_x^2 + \frac{C_\beta}{x^2} + \omega\,x^{\beta}
\end{equation}
is a singular Sturm-Liouville operator on $(0,1)$, for every $\omega\geq 0$.
Define the parameters
\begin{equation}\label{eq:nu-kappa}
	\nu = \frac{1}{2} + \frac{\beta n}{4}, 
	\qquad 
	\kappa = \frac{2}{\beta+2}\quad\Big(\text{equivalently } \alpha=\frac{2\beta}{\beta+2},\ \ \kappa=1-\frac{\alpha}{2}\Big).
\end{equation}
Note that $C_\beta = \nu^2-\frac{1}{4}$ and thus, near $x=0$, $P_\omega$ is a Bessel operator of index $\nu$ independent of~$\omega$.

\subsection{Analysis of the Friedrichs boundary condition at $x=0$}\label{sec:friedrichs}
We fix the boundary condition at the singular point $x=0$ by choosing the \emph{Friedrichs extension} of the minimal operator associated with $P_\omega$, for any $\omega\geq 0$.

\paragraph{Quadratic form and Friedrichs extension.}
Given any fixed $\omega\geq 0$, let $P_{\min}$ be the symmetric operator on $L^2((0,1),dx)$ with domain $C_c^\infty(0,1)$. Its natural, lower-bounded quadratic form is
\begin{equation}\label{eq:qform-app}
	\mathfrak q_\omega[u] = \int_0^1 \Big( |u'(x)|^2 + \frac{\nu^2-\frac{1}{4}}{x^2}\,|u(x)|^2 + \omega x^\beta |u(x)|^2 \Big)\,dx 
	\qquad \forall u\in C_c^\infty(0,1).
\end{equation}
Let $\overline{\mathfrak q_\omega}$ be the closure of $\mathfrak q_\omega$ in the norm $\|u\|_{\mathfrak q}^2 = \mathfrak q_\omega[u]+\|u\|_{L^2(0,1)}^2$. By the Kato-Lions-Milgram-Nelson theorem~\cite{Davies1989,Kato1995,RS1975},
the closed form $\overline{\mathfrak q_\omega}$ is densely defined and lower semicontinuous; the associated self-adjoint operator is the \emph{Friedrichs extension} $P_F$ of $P_{\min}$. This is what we mean by ``Friedrichs boundary condition at $x=0$''.

To apply it here, take
$$
q_0[u]=\int_0^1\left(|u'(x)|^2+\frac{\nu^2-\frac{1}{4}}{x^2}|u(x)|^2\right)dx,\qquad
v[u]=\int_0^1\omega\,x^\beta |u(x)|^2\,dx .
$$
Since $x^\beta\in L^\infty(0,1)$, we have $|v[u]|\leq \omega\|x^\beta\|_{L^\infty}\,\|u\|_{L^2}^2$, hence the relative bound is $a=0$. Therefore $q=q_0+v$ is closed and semi-bounded on $D(q_0)$, and the associated Friedrichs operator is precisely $P_\omega$ with the Friedrichs boundary condition at $x=0$.

\paragraph{Asymptotics at $x=0$ and the boundary constraint.}
We first recall what is Frobenius analysis near a regular singular point (for standard references see~\cite{CL1955,Olver1974,Teschl2012,Zettl2005}).
For a second–order ODE
$$
y''+p(x)y'+q(x)y=0
$$
with a \emph{regular singular point} at $x=0$
(i.e., $p_0(x) = x p(x)$ and $q_0(x) = x^2q(x)$ admit finite limits as $x\rightarrow 0^+$ and are analytic in a neighborhood),
the \emph{Frobenius method} seeks solutions of the form
$$
y(x)=x^{r}\sum_{m=0}^{\infty}a_mx^m
$$ 
with $a_0\neq 0$,
where the exponent $r$ solves the \emph{indicial equation}
$$
r(r-1)+p_0(0) r+q_0(0)=0.
$$
If the two roots $r_1\neq r_2$ do not differ by an integer, one obtains two independent series solutions
$x^{r_1}(1+\cdots)$ and $x^{r_2}(1+\cdots)$. In the resonant case ($r_1-r_2\in\N$), the second solution may contain a $\log x$ factor.

\medskip

In our setting, let us perform the Frobenius analysis of $(P_\omega-\lambda)u=0$ near $x=0$, where $P_\omega$ is defined by \eqref{eq:Pomega-app}, of associated Friedrichs form on $L^2((0,1),dx)$ given by \eqref{eq:qform-app}.
The point $x=0$ is regular singular, the lower-order term $\omega x^\beta$ ($\beta>0$) does not affect the indicial equation, and we obtain
$r_\pm=\frac{1}{2}\pm\nu$. Thus
\begin{equation}\label{eq:u_asympt}
	u(x)=A\,x^{\frac{1}{2}+\nu}(1+\mathrm{o}(1))+B\,x^{\frac{1}{2}-\nu}(1+\mathrm{o}(1))
\end{equation}
as $x\rightarrow 0^+$, with constants $A,B\in\C$ independent of $x$ (and uniform in $\omega$ for fixed $\nu,\beta$),
with a possible $\log x$ factor in the singular branch when $2\nu\in\N$.

The \emph{Friedrichs} extension is characterized by the boundary condition at $x=0$ that kills the non-energy-finite (more singular) Frobenius branch. Here, the \emph{Friedrichs} boundary condition at $x=0$ selects the regular branch ($B=0$).
Indeed, in the integral \eqref{eq:qform-app}, the ``Hardy term"
$x^{-2}|u(x)|^2$ with $u$ given by \eqref{eq:u_asympt} contributes
$$
x^{-2}|u(x)|^2\sim \begin{cases}
	|B|^2 x^{-1-2\nu} & \text{if }B\neq 0,\\[2pt]
	|A|^2 x^{-1+2\nu} & \text{if }B=0,
\end{cases}
$$
and thus $\int_0^1 x^{-2}|u(x)|^2\,dx$ converges only if $B=0$. 
Hence, being in the Friedrichs domain means that $B=0$, which is equivalent to the trace condition
\begin{equation}\label{friedrichs_condition}
	\boxed{
		\lim_{x\rightarrow 0^+} x^{-\frac{1}{2}+\nu} u(x)=0 
	}
\end{equation}
This boundary constraint characterizes the Friedrichs extension at $x=0$.

\begin{remark}[Relation with limit-point / limit-circle.]
	The endpoint $x=0$ is a
	\begin{itemize}
		\item \emph{limit-point} if $\nu\geq 1$: there is a unique self-adjoint realization (no boundary condition to impose at $0$); the singular branch is automatically ruled out by square-integrability of the energy density;
		\item \emph{limit-circle} if $\frac{1}{2}<\nu<1$: self-adjoint extensions are parameterized by a boundary condition at $0$; the Friedrichs extension corresponds to \eqref{friedrichs_condition}.
	\end{itemize}
	In both cases the Friedrichs realization is the one determined by the closed quadratic form \eqref{eq:qform-app}.
\end{remark}

\paragraph{Relevance for the wave equation and boundary observability.}
For the wave equation $\partial_t^2 u+P_\omega u=0$ the Friedrichs realization (in particular, the boundary condition \eqref{friedrichs_condition}) guarantees that:
\begin{enumerate}
	\item $P_\omega$ is self-adjoint and positive on $L^2((0,1),dx)$, so the Cauchy problem is well-posed and the spectral calculus applies.
	\item The eigenfunctions select the regular Frobenius branch at $x=0$, which ensures that the normal trace $\lim_{x\rightarrow 0^+}x^{\frac{1}{2}-\nu}\partial_x u$ is finite and well-defined (see Section~\ref{sec:neumann-obs}).
	\item The choice is \emph{uniform in $\omega$} (tangential frequency), which is crucial for transferring the one-dimensional observability to the separable multi-dimensional setting.
\end{enumerate}

Note that the conjugation has changed $\widetilde\triangle_g$ into $\triangle_g$ whose Friedrichs form is equivalent to that of $\triangle_g$. In particular, the selection of the regular branch in the Frobenius analysis is invariant under conjugation.

Since $u=mz\sim x^{-\frac{\beta n}{4}}z=x^{\frac{1}{2}-\nu}z$ near $x=0$, in the original variable $z$, the Friedrichs condition~\eqref{friedrichs_condition} is 
$$
\boxed{
	\lim_{x\rightarrow 0^+} z(x)=0
}
$$

\subsection{Definition of the Neumann-like observation trace}
\label{sec:neumann-obs}
We now claim that, in the $u$ variable, the good observation corresponding to \eqref{physical_flux} is
\begin{equation}\label{conj_obs}
	\boxed{
		\mathcal{C}_{\textrm{conj}}(u)=\lim_{x\rightarrow 0^+}x^{\frac{1}{2}-\nu}\partial_x u
	}
\end{equation}

\begin{lemma}\label{lem_trace}
	Let $u$ be in the Friedrichs domain of $\triangle_g$ on $L^2((0,1)\times M,dx\,dv_G)$.
	Then, as $x\rightarrow 0^+$,
	\begin{equation}\label{eq:frobenius-u}
		u(x,y) = x^{\frac{1}{2}+\nu} A(y)\ +\ x^{\frac{1}{2}+\nu+\rho} B(x,y),
		\qquad \nu=\frac{1}{2}+\frac{\beta n}{4},\ \ \rho>0,
	\end{equation}
	with $A\in L^2(M,dv_G)$ and $B\in C^\infty([0,1)\times M)$ bounded with values in $L^2(M,dv_G)$.
	Moreover, the following limit exists in $L^2(M,dv_G)$ and coincides with the leading coefficient:
	\begin{equation}\label{eq:ren-trace}
		\lim_{x\rightarrow 0^+}x^{\frac{1}{2}-\nu}\partial_x u(x,\cdot)=(\nu+\frac{1}{2})\,A(\cdot) .
	\end{equation}
	Consequently, recalling that $z$ is the original (geometric) unknown, related to $u$ by $u=m z$ with $m=x^{-\frac{\beta n}{4}}\rho$, we have
	\begin{equation}\label{eq:constant}
		\boxed{\ \mathcal{C}_{\textrm{phys}}(z)
			=\frac{1+\frac{\beta n}{2}}{\ \nu+\frac{1}{2}\ }\ \mathcal{C}_{\textrm{conj}}(u)\ .\ }
	\end{equation}
	In particular, $\mathcal{C}_{\textrm{phys}}$ and $\mathcal{C}_{\textrm{conj}}$ are equivalent up to the positive constant $\frac{1+\beta n/2}{\nu+1/2}=\frac{2\nu}{\nu+1/2}$.
\end{lemma}

\begin{proof}
	We use the expansion
    $$
    u(x,y)=\sum_{j\ge0}u_j(x)\,\psi_j(y)
    $$
    The normal form of $\triangle_g$ near $x=0$ is
    $$
    \triangle_g=-\partial_x^2+\frac{\nu^2-\frac{1}{4}}{x^2}+x^\beta\triangle_G+\mathrm{O}(x)
    $$
	Projecting onto $\psi_j$ gives near $x=0$ a the regular singular ODE
	$$
	-u_j''(x)+\frac{\nu^2-\frac{1}{4}}{x^2}u_j(x)+\omega_j x^\beta u_j(x)=\mathrm{O}(x).
	$$
	The Frobenius theory recalled in the previous section gives (the $O(x)$ coefficients do not alter the indicial equation)
	$$
	u_j(x)=A_j x^{\frac{1}{2}+\nu} + x^{\frac{1}{2}+\nu+\rho}B_j(x),\qquad \rho>0,
	$$
	with $B_j$ smooth and bounded (the Friedrichs domain excludes the singular branch $x^{\frac{1}{2}-\nu}$).
	Set
    $$
    A(y)=\sum_j A_j \psi_j(y)
    \qquad\text{and}\qquad
    B(x,y)=\sum_j B_j(x)\psi_j(y)
    $$
    with $A\in L^2(M,dv_G)$ and $B$ bounded in $L^2(M,dv_G)$.
	We have obtained \eqref{eq:frobenius-u}.
	
	Now, we compute
	$$
	\partial_x u(x,y)=(\nu+\frac{1}{2})\,A(y)\,x^{-\frac{1}{2}+\nu}+x^{-\frac{1}{2}+\nu+\rho}\,C(x,y),
	$$
	with $C\in L^2(M,dv_G)$.
	Multiplying by $x^{\frac{1}{2}-\nu}$ gives
	$$
	x^{\frac{1}{2}-\nu}\partial_x u(x,\cdot)=(\nu+\frac{1}{2})A(\cdot)+x^{\rho}C(x,\cdot).
	$$
	Since $\rho>0$ and $C\in L^2$, the remainder tends to $0$ in $L^2(M,dv_G)$ as $x\rightarrow 0^+$, hence \eqref{eq:ren-trace}.
\end{proof}

\begin{remark}
	When $M$ is a point, we have $n=0$ so that $X=(0,1)$ and we recover the 1D case studied in Section~\ref{sec:1D-case}, with $\nu=\frac{1}{2}$. In this case, 
	$$
	\mathcal{C}_{\textrm{conj}}=\partial_x u(0)=\mathcal{C}_{\textrm{phys}}z =\partial_x z(0)
	$$
	i.e., the observation is the classical Neumann trace.
\end{remark}

\subsection{Boundary observability inequality}
\label{sec:boundary-obs-inequality}

\subsubsection{``Anisotropic" Sobolev spaces}
For $s\in\R$, let $H^{s}_x(0,1)$ be the (Bessel) Sobolev scale along the normal variable defined in Section~\ref{sec:1D-case} via the spectral powers of the one-dimensional Bessel operator (equivalently, using the basis from Proposition~\ref{prop:bessel-functions}). Define the ``anisotropic" Sobolev spaces
\begin{equation}\label{eq:anisotropic}
	H^{s}_x L^2(M) 
	= \Big\{ F\in L^2\big((0,1)\times M\big) \ \mid\ F(\cdot,y)=\sum_{j\geq 1} f_j(\cdot)\,\psi_j(y),\ \ \sum_{j\geq 1}\|f_j\|_{H^s_x(0,1)}^2<\infty \Big\},
\end{equation}
with the obvious norm. By functional calculus on $M$, for $\theta\geq 0$ we further set
\begin{equation}\label{eq:anisotropic-omega}
	\|(\triangle_G^\theta) F\|_{H^{s}_x L^2(M)}^2 = \sum_{j\geq 1} \omega_j^{2\theta}\,\|f_j\|_{H^{s}_x(0,1)}^2.
\end{equation}

\subsubsection{Proof of the main theorem}

Let us recall the setting of the main theorem (Theorem~\ref{thm:obs-inequality}). Let $u$ solve \eqref{eq:wave} with Friedrichs boundary condition at $x=0$ and either Dirichlet or Neumann at $x=1$.  The claim is that for every $T > T^* = 2/\kappa = \beta + 2$, there exists constants $0 < c_T \leq C_T$ such that
\begin{equation}
	c_T\,\mathcal{E}_\nu[u_0,u_1]
	\ \leq\ 
	\int_0^T\int_M \Big\vert \lim_{x\rightarrow 0^+}x^{\frac{1}{2}-\nu}\partial_x u(x,y,t)\Big\vert^2\, dv_G(y)\,dt
	\ \leq\ 
	C_T\,\mathcal{E}_\nu[u_0,u_1].
\end{equation}
The anisotropic energy is defined using the function spaces~\eqref{eq:anisotropic} by
\begin{equation}
		\displaystyle 
		\mathcal{E}_\nu[u_0,u_1]
		= \|u_0\|_{H^{\nu+\frac{1}{2}}_x L^2(M)}^2
		+ \|u_1\|_{H^{\nu-\frac{1}{2}}_x L^2(M)}^2
		+ \big\|\triangle_G^{1/2} u_0 \big\|_{H^{\nu-\frac{1}{2}}_x
		L^2(M)}^2.
\end{equation}

Before moving to the proof of the main theorem, we some comments are in order.

\begin{remark}\label{rem_quasiisom}
	In \eqref{energies}, it is understood that $L^2(M)$ is considered with respect to the volume $dv_G$.
	By quasi-isometry near $x=0$ (see~\cite[Remark~1 and Appendix A.1]{CdVDdHT2024}), the choice of the tangential metric $G=g_1(0)$ in \eqref{energies} is not important: replacing $\triangle_G$ by $\triangle_{g_1(x_0)}$ for any fixed $x_0\in[0,1)$ yields an equivalent energy.
	
	This indeed follows from Lemma \ref{lem:equiv-frozen-metric} given in Appendix \ref{app:equiv-frozen-metric}.
	More precisely, fix any $x_0\in[0,1)$ and set $h_0=G=g_1(0)$, $h_1=g_1(x_0)$. Since $g_1(\cdot)$ is $C^\infty$ in $x$ and $M$ is compact, we have \eqref{eq:qi}--\eqref{eq:vol-comp} in Lemma \ref{lem:equiv-frozen-metric} with constants independent of $x_0$ in compact subsets of $[0,1)$.
	Taking $s=\nu-\frac{1}{2}$ in \eqref{eq:vector-valued} yields
	$$
	\big\|\triangle_g^{1/2}u_0\big\|_{H^{\nu-\frac{1}{2}}_x L^2(M,dv_g)}
	\ \asymp\
	\big\|\triangle_{g_1(x_0)}^{1/2}u_0\big\|_{H^{\nu-\frac{1}{2}}_x L^2(M,dv_{g_1(x_0)})},
	$$
	and similarly the zero-th order terms in $\mathcal{E}_\nu$ are equivalent by \eqref{eq:form-equivalence} with $\nabla\varphi\equiv 0$. Hence the energy in \eqref{energies} is unchanged up to multiplicative constants when replacing $\triangle_g$ by $\triangle_{g_1(x_0)}$.
\end{remark}

\begin{remark}
	The observability inequality \eqref{main_obs} is written in the $u$ coordinate. Let us express it in the initial (geometric) $z$ coordinate.
	Recall that $u=Uz=mz$ with $m=x^{-\frac{\beta n}{4}}\rho$ (unitary map $U:L^2(X,dv_g)\to L^2((0,1)\times M,dx\,dv_G)$, $\rho(0,\cdot)=1$), that
	\begin{align*}
		\mathcal{C}_{\mathrm{conj}}u &=\lim_{x\rightarrow 0^+}x^{-(\nu-\frac{1}{2})}\partial_x u\in L^2(M,dv_G), \\
		\mathcal{C}_{\mathrm{phys}}z &=\lim_{x\rightarrow 0^+}|\det g|^{1/2}g^{xx}\partial_x z\in L^2(M,dv_G),
	\end{align*}
	with $\nu=\frac{1}{2}+\frac{\beta n}{4}$, and that $\mathcal{C}_{\mathrm{phys}}z=\frac{1+\frac{\beta n}{2}}{\ \nu+\frac{1}{2}\ }\;\mathcal{C}_{\mathrm{conj}}u$.
	By unitarity of $U$ and $\triangle_g=U\widetilde\triangle_gU^{-1}$, the energies coincide:
	\begin{equation}\label{eq:energy-unitary}
		\mathcal{E}_\nu[u_0,u_1]=\widetilde{\mathcal{E}}_\nu[z_0,z_1],
		\qquad (u_0,u_1)=U(z_0,z_1),
	\end{equation}
	where $\mathcal{E}_\nu$ is the (product-measure) anisotropic energy used in the conjugated gauge and $\widetilde{\mathcal{E}}_\nu$ is the geometric Friedrichs energy of $\widetilde\triangle_g$ in $L^2(X,dv_g)$.
	
	Consequently, the observability inequality proved in the conjugated gauge (for $T>2/\kappa$)
	$$
	c_T\,\mathcal{E}_\nu[u_0,u_1]
	\ \leq\ \int_0^T\!\|\mathcal{C}_{\mathrm{conj}}u(\cdot,t)\|_{L^2(M,dv_G)}^2\,dt
	\ \leq\ C_T\,\mathcal{E}_\nu[u_0,u_1]
	$$
	is equivalent to the \emph{physical} inequality
	$$
	c_T^{\mathrm{phys}}\,\widetilde{\mathcal{E}}_\nu[z_0,z_1]
	\ \leq\ \int_0^T\!\|\mathcal{C}_{\mathrm{phys}}z(\cdot,t)\|_{L^2(M,dv_G)}^2\,dt
	\ \leq\ C_T^{\mathrm{phys}}\,\widetilde{\mathcal{E}}_\nu[z_0,z_1]
	$$
	with $c_T^{\mathrm{phys}} =\Big(\frac{1+\beta n/2}{\nu+1/2}\Big)^2\,c_T$ and $C_T^{\mathrm{phys}}
	=\Big(\frac{1+\beta n/2}{\nu+1/2}\Big)^2\,C_T$.
\end{remark}

For clarity, and because the arguments are of different natures, we first make the proof in the separable case, and then in the general, non-separable case. 

\subsubsection{Proof of Theorem \ref{thm:obs-inequality} in the separable case}\label{sec_proof_main_thm_separable}
In the separable case, the proof consists of a uniform-in-$k$ reduction to Theorem~\ref{thm:1d-observability} of Section~\ref{sec:1D-case}.

\smallskip
\noindent\emph{Modal reduction.}
Expanding $u(t,x,y)=\sum_{k\geq 1} u_k(t,x)\psi_k(y)$, we have
\begin{equation}\label{eq:modal-wave}
	\partial_t^2 u_k + P_{\omega_k} u_k = 0,
	\qquad 
	P_{\omega} = -\partial_x^2 + \frac{\nu^2-\frac{1}{4}}{x^2} + \omega x^{\beta}\quad \text{on }(0,1),
\end{equation}
with the same boundary condition at $x=1$ and the Friedrichs condition at $x=0$. Let $\{\lambda_n(\omega)\}_{n\geq 1}$ be the eigenvalues of $P_{\omega}$ (increasing order), and denote $\mu_n(\omega)=\sqrt{\lambda_n(\omega)}$. 
The corresponding (real) eigenfunctions $(\varphi_{n,\omega})_{n\geq 1}$ have the Frobenius asymptotics, uniform in $\omega$:
$$
\varphi_{n,\omega}(x)=A_{n,\omega}\,x^{\frac{1}{2}+\nu}\big(1+\mathrm{O}(x)\big),
\qquad
x^{\frac{1}{2}-\nu}\partial_x\varphi_{n,\omega}(x)\to(\nu+\frac{1}{2})A_{n,\omega},
$$
as $x\rightarrow 0^+$. Hence
\begin{equation}\label{eq:ren-time-trace}
	\mathcal{C}_{\mathrm{conj}}u_k(t)
	=\sum_{n\geq 1} a_{n,k}\,e^{i\mu_n(\omega_k)t}
	+\sum_{n\geq 1}\overline{a_{n,k}}\,e^{-i\mu_n(\omega_k)t},
\end{equation}
where the coefficients $a_{n,k}$ are linear in the spectral coefficients of the initial data $(u_{0,k},u_{1,k})$ against $\{\varphi_{n,\omega_k}\}$. The proportionality constants are \emph{uniform} in $\omega$.
By the Parseval theorem,
\begin{equation}\label{eq:time-trace}
	\int_0^T\!\!\int_M \big|\mathcal{C}_{\mathrm{conj}}u(\cdot,t)\big|^2\,dv_G\,dt
	= \sum_{k\geq 1} \int_0^T \left|\sum_{n\in\mathbb Z} a_{n,k}\, e^{i\tilde\mu_n(\omega_k) t}\right|^2 dt,
\end{equation}
where $\tilde\mu_{\pm n}(\omega)=\pm\mu_n(\omega)$.

\smallskip\noindent\emph{Ingham step (uniform constants).}
For each fixed $\omega\geq 0$, the nonharmonic set $\{\tilde\mu_n(\omega)\}_{n\in\mathbb{Z}}$ satisfies the uniform gap condition. Indeed, to be very precise, it has the Beurling upper density
$D^+=\frac{1}{\pi\kappa}$, identical to the case $\omega=0$. This follows from the Weyl law for singular Sturm--Liouville operators: adding the bounded potential $\omega x^\beta$ does not change the linear asymptotics $\mu_n(\omega)\sim \kappa\pi n$ as $n\to\infty$.
Hence, by the Ingham theorem (see Remark~\ref{rem:ingham})
for every $T>2\pi D^+=\frac{2}{\kappa}$ there exists $c_T>0$ (depending on $T$ and $\kappa$ but \emph{independent} of $\omega$) such that
$$
c_T \sum_{n\in\mathbb{Z}} |a_{n,k}|^2
\ \leq\ 
\int_0^T \left|\sum_{n\in\mathbb{Z}} a_{n,k}\, e^{i\tilde\mu_n(\omega_k) t}\right|^2 dt
\ \leq\ 
C_T \sum_{n\in\mathbb{Z}} |a_{n,k}|^2,
\qquad \forall k\geq 1.
$$
Summing over $k$ and using \eqref{eq:time-trace} yields
\begin{equation}\label{eq:frame}
	\int_0^T\!\!\int_M \big|\mathcal{C}_{\mathrm{conj}}u(\cdot,t)\big|^2\, dv_G\,dt
	\ \asymp_T\ \sum_{k\geq 1}\sum_{n\in\mathbb{Z}} |a_{n,k}|^2 .
\end{equation}
In particular, in \eqref{eq:frame} the constants are \emph{uniform} with respect to $\omega_k$, $n$ and the initial data.

\smallskip\noindent
\emph{Identification of the energy (uniform in $k$).}
Exactly as in Section~\ref{sec:1D-case} (Theorem~\ref{thm:1d-observability}), one has uniformly in $\omega$ that
$$
\sum_{n\in\Z} |a_{n,k}|^2 \ \asymp\  \|f_{0,k}\|_{H^{\nu+\frac{1}{2}}_x(0,1)}^2 + \|f_{1,k}\|_{H^{\nu-\frac{1}{2}}_x(0,1)}^2 + \omega_k\,\|f_{0,k}\|_{H^{\nu-\frac{1}{2}}_x(0,1)}^2,
$$
where $(f_{0,k},f_{1,k})$ are the modal initial data of $u$.
Summing over $k$ and using \eqref{eq:anisotropic}-\eqref{eq:anisotropic-omega} gives
$$
\sum_{k\geq 1}\sum_{n\in\Z} |a_{n,k}|^2 \ \asymp\  \|u_0\|_{H^{\nu+\frac{1}{2}}_x L^2(M)}^2 + \|u_1\|_{H^{\nu-\frac{1}{2}}_x L^2(M)}^2 + \big\|\triangle_G^{1/2} u_0 \big\|_{H^{\nu-\frac{1}{2}}_x L^2(M)}^2.
$$
Combining with \eqref{eq:frame} proves \eqref{main_obs}. The optimality of $T_\ast$ follows from the density computation (as in Section~\ref{sec:1D-case}, Remark~\ref{rem:upper-density}).

This finishes the proof of the theorem in the separable case.

\begin{remark}
	The only point where the dependence on $\omega_k$ matters is the Ingham step. The crucial facts are: 
	\begin{itemize}
		\item near $x=0$ the Bessel index $\nu$ is \emph{independent of $\omega$}, so the boundary coupling coefficients $a_{n,k}$ have the same structure as in Section~\ref{sec:1D-case}; 
		\item the spectral counting function for $P_\omega$ satisfies $N(\mu;\omega)\sim \frac{\mu}{\pi\kappa}$ as $\mu\to\infty$, uniformly in $\omega\geq 0$, hence the upper density $D^+$ and the threshold $T_\ast=2\pi D^+=2/\kappa$ are uniform in~$\omega$.
	\end{itemize}
\end{remark}

\begin{remark}
	Writing $u_j=\sum_{k\geq 1} f_{j,k}\psi_k$ in \eqref{energies} recovers the energy
	$$
	\sum_{k\geq 1}\Big(\|f_{0,k}\|_{H_x^{\nu+\frac{1}{2}}}^2
	+\|f_{1,k}\|_{H_x^{\nu-\frac{1}{2}}}^2
	+\omega_k\,\|f_{0,k}\|_{H_x^{\nu-\frac{1}{2}}}^2\Big),
	$$
	which is the expansion used in the separable analysis (and is what we have computed in the proof).
\end{remark}

\subsubsection{Proof of Theorem \ref{thm:obs-inequality} in the non-separable case}
\label{sec:proof-non-separable}
The non-separable case can be tackled as a perturbation of the separable model. 
Recall that $\triangle_g = \triangle_g^{\textrm{sep}} + O(x)$.
More precisely, we write the full operator as a perturbation of $\triangle_g^{\textrm{sep}}$:
$$
\triangle_g = \triangle_g^{\textrm{sep}} + K,\qquad K=\mathcal{T}(x,y,D_y) + x^{\beta}\big(\triangle_{G(x)}-\triangle_G\big) + \mathcal{R}(x,y,D_x ,D_y),
$$
where $\mathcal{T}$ is tangential of order $1$, $x^\beta(\triangle_{G(x)}-\triangle_G)$ is tangential of order $2$ with $C^1$-coefficients $O(x)$, and $\mathcal{R}$ collects a $0$-th order potential, mixed terms $\partial_xD_y$, and a possible first-order $\partial_x$ part,
all with $C^1$-coefficients $O(x)$.
Set
$$
\mathcal{H}=H^{\nu+\frac{1}{2}}_x L^2(M)\times H^{\nu-\frac{1}{2}}_x L^2(M),
\qquad 
\|(z_0,z_1)\|_{\mathcal{H}}^2 =\|z_0\|_{H^{\nu+\frac{1}{2}}_x L^2(M)}^2  +\|z_1\|_{H^{\nu-\frac{1}{2}}_x L^2(M)}^2,
$$
and denote, for short, by $\mathcal{C}$ the admissible observation operator (for reminders on admissibility, see~\cite{Trelat2024,TW2009}). 
Let
$$
\mathcal{A}_0=\begin{pmatrix}0&I\\-\triangle_g^{\textrm{sep}}&0\end{pmatrix}
$$
be the wave operator (with Friedrichs realization at $x=0$), of domain $D(\mathcal{A}_0)$, generating the wave group in the separable case.
By Theorem~\ref{thm:obs-inequality}, the pair $(\mathcal{A}_0,\mathcal{C})$ is exactly observable on $[0,T]$ for every $T>2/\kappa$, i.e., 
\begin{equation}\label{eq:obs-A0}
	\exists\,c_T>0\ \mid\ \forall Z_0\in\mathcal{H}\qquad \int_0^T\|\mathcal{C} e^{t\mathcal{A}_0}Z_0\|_{L^2(M)}^2\, dt \geq c_T\,\|Z_0\|_{\mathcal{H}}^2.
\end{equation}

Let us now develop the perturbation argument.
Fix a smooth cutoff $\chi\in C_c^\infty([0,2\delta))$ with $\chi\equiv1$ on $[0,\delta]$ and set
\begin{equation}\label{eq:K-split}
	K_{\mathrm{near}}=\chi(x)\,K,\qquad K_{\mathrm{far}}=(1-\chi(x))\,K,\qquad K=K_{\mathrm{near}}+K_{\mathrm{far}}.
\end{equation}
By Lemma \ref{lem_normal_form} and the equivalence of tangential energies (Remark \ref{rem_quasiisom} and Lemma \ref{lem:equiv-frozen-metric}), $K_{\mathrm{near}}:D(\mathcal{A}_0)\rightarrow\mathcal{H}$ is bounded and
\begin{equation}\label{eq:near-small}
	\|K_{\mathrm{near}}\|_{\mathcal{L}(D(\mathcal{A}_0),\mathcal{H})}\leq C\,\delta,
\end{equation}
for some $C>0$ not depending on $\delta$.
On the other hand, $K_{\mathrm{far}}$ is supported in $\{x\geq\delta\}$ where all coefficients are smooth and bounded; there, the embedding $D(\mathcal{A}_0)\hookrightarrow \mathcal{H}$ is compact in the $x$-variable on $(\delta,1)$, hence $K_{\mathrm{far}}:D(\mathcal{A}_0)\rightarrow \mathcal{H}$ is a compact operator.

\medskip

We now perturb in two steps.

First, we use the fact that exact observability is stable under any compact perturbation 
(see~\cite[Theorem 1.1]{DO2018}). 
Since $K_{\mathrm{far}}:D(\mathcal{A}_0)\rightarrow \mathcal{H}$ is compact and $\mathcal{C}$ is admissible for $\mathcal{A}_0$, 
and since the Fattorini--Hautus property holds for $(\mathcal{A}_1^*,\mathcal{C})$ by Lemma~\ref{lem:FH-compact} hereafter,
there exists $c_T'>0$ such that
\begin{equation}\label{eq:obs-A1}
	\int_0^T\|\mathcal{C} e^{t(\mathcal{A}_0+\mathcal{B})}Z_0\|_{L^2(M)}^2dt
	\geq c_T'\,\|Z_0\|_{\mathcal{H}}^2
	\qquad\text{for }\ \mathcal{B}=\begin{pmatrix}0&0\\-K_{\mathrm{far}}&0\end{pmatrix}.
\end{equation}

\begin{lemma}[Fattorini--Hautus near the singular boundary: compactly perturbed case]\label{lem:FH-compact}
	Let 
	$$
	\mathcal{A}_0=\begin{pmatrix}0&I\\-\triangle_g^{\textrm{sep}}&0\end{pmatrix}, 
	\qquad
	\mathcal{A}_1=\begin{pmatrix}0&I\\-\triangle_g^{\textrm{sep}}-K_{\mathrm{far}}&0\end{pmatrix},
	$$
	where $K_{\mathrm{far}}$ is supported in $\{x\geq\delta\}$ and has smooth bounded coefficients.
	Then
	$$
	\ker(\lambda I-\mathcal{A}_1^*)\cap\ker\mathcal{C} = \{0\} \qquad\forall \lambda\in\C.
	$$
\end{lemma}

\begin{proof}
	Set $L_1=\triangle_g^{\textrm{sep}}+K_{\mathrm{far}}$. Because $K_{\mathrm{far}}$ is supported in $\{x\ge\delta\}$, we have $L_1=\triangle_g^{\textrm{sep}}$ on the collar $\{0<x<\delta\}$.
	
	Let $(\phi,\psi)\in\ker(\lambda I-\mathcal{A}_1^*)\cap\ker\mathcal{C}$. Then $\psi=i\lambda\,\phi$, $(\triangle_g^{\textrm{sep}}+K_{\mathrm{far}})\phi=\lambda^2\phi$ with the Friedrichs boundary condition at $x=0$, and $\mathcal{C}(\phi,i\lambda\phi)=0$.
	Let $(\psi_j)_{j\in\N}$ be an orthonormal basis of $L^2(M)$ consisting of eigenvectors of $\triangle_G$, i.e., $-\triangle_G \psi_j=\mu_j^2 \psi_j$ for any $j\in\N$. On $(0,\delta)\times M$ we expand $\phi(x,y)=\sum_{j\geq 0} u_j(x)\psi_j(y)$ and, since $K_{\mathrm{far}}=0$ on $(0,\delta)$, each $u_j$ satisfies
	\begin{equation}\label{eq:FHcompact-ode}
		-u_j''(x)+\frac{\nu^2-\frac{1}{4}}{x^2}\,u_j(x)+\mu_j^2\,x^\beta u_j(x)=\lambda^2 u_j(x)\qquad\forall x\in(0,\delta),
	\end{equation}
	with Friedrichs condition at $x=0$.
	Following the reminders done in Section \ref{sec:friedrichs}, $x=0$ is a regular singular point, its indicial roots are $r_\pm=\frac{1}{2}\pm\nu$. The Friedrichs condition excludes the singular branch $x^{\frac{1}{2}-\nu}$.
	Hence, for every $j\in\N$, $u_j(x)=A_jx^{\frac{1}{2}+\nu}(1+\mathrm{o}(1))$ as $x\rightarrow 0^+$ for some $A_j\in\C$.
	
	We obtain
	$$
	\mathcal{C}\big(u_j \psi_j,\, i\lambda u_j \psi_j\big)=c_{\beta,\nu}\,A_j\,\psi_j\in L^2(M).
	$$
	Summing in $j$ (the series is convergent in the collar),
	$$
	\mathcal{C}(\phi,i\lambda\phi)=c_{\beta,\nu}\,\sum_{j\ge0} A_j\,\psi_j.
	$$
	Since $\mathcal{C}(\phi,i\lambda\phi)=0$ and $\{\psi_j\}$ is orthonormal, we conclude
	$$
	A_j=0\quad\forall j.
	$$
	Finally, $u_j\equiv0$ on $(0,\delta)$ for all $j$, hence $\phi\equiv0$ on the collar $(0,\delta)\times M$.
	Since $L_1$ has smooth coefficients on $[\delta,1]\times M$, uniqueness for second-order elliptic equations (or, equivalently, ODE uniqueness mode-by-mode at $x=\delta$) yields $\phi\equiv0$ on $(0,1)\times M$.
	Thus $(\phi,\psi)=(0,0)$ and $\ker(\lambda I-\mathcal{A}_1^*)\cap\ker\mathcal{C}=\{0\}$.
\end{proof}

Second, we use the fact that exact observability is stable under \emph{small} bounded perturbations (see~\cite[Chapter 6, Section 6.3]{TW2009}). 
No additional Fattorini-Hautus check is required in this step; nevertheless, Lemma~\ref{lem:FH-near} hereafter shows that the Fattorini-Hautus property persists for $(\mathcal{A}_g^*,\mathcal{C})$.
We choose $\delta>0$ small enough such that
\begin{equation}\label{eq:small-bounded}
	\|\mathcal{B}_{\mathrm{near}}\|_{\mathcal{L}(D(\mathcal{A}_0),\mathcal{H})}\leq \varepsilon_T
	\qquad\text{for}\quad 
	\mathcal{B}_{\mathrm{near}}=\begin{pmatrix}0&0\\-K_{\mathrm{near}}&0\end{pmatrix},
\end{equation}
with $\varepsilon_T$ small enough to preserve exact observability on $[0,T]$. Combining \eqref{eq:obs-A1} and \eqref{eq:small-bounded} yields the desired observability estimate 
$$
\int_0^T\|\mathcal{C} e^{t\mathcal{A}_g}Z_0\|_{L^2(M)}^2\, dt \geq c_T''\,\|Z_0\|_{\mathcal{H}}^2\qquad (T>2/\kappa).
$$
with
$$
\mathcal{A}_g=\mathcal{A}_0+\mathcal{B}+\mathcal{B}_{\mathrm{near}}
=\begin{pmatrix}0&I\\-(\triangle_g^{\textrm{sep}}+K_{\mathrm{far}}+K_{\mathrm{near}})&0\end{pmatrix}
=\begin{pmatrix}0&I\\-\triangle_g&0\end{pmatrix}.
$$

\begin{lemma}[Fattorini--Hautus stability under small collar perturbations]\label{lem:FH-near}
	Let $\mathcal{A}_1$ be as above and set
    \begin{equation}
    \mathcal{A}_g=\begin{pmatrix}0&I\\-(\triangle_g^{\textrm{sep}}+K_{\mathrm{far}}+K_{\mathrm{near}})&0\end{pmatrix}
    \end{equation}
	where $K_{\mathrm{near}}$ is supported in $\{x\le 2\delta\}$ and its coefficients are $O(x)$ in $C^1(M)$ with $\|K_{\mathrm{near}}\|$ arbitrarily small (by shrinking $\delta$).
	Then, for every $\lambda\in\C$,
	$$
	\ker(\lambda I-\mathcal{A}_g^*)\cap\ker\mathcal{C} = \{0\}.
	$$
\end{lemma}

\begin{proof}
	Near $x=0$, the principal part of $\triangle_g^{\textrm{sep}}+K_{\mathrm{near}}$ is still the Bessel operator $-\partial_x^2+\frac{C_\beta}{x^2}$; the additional terms are $O(x)$ (in $C^1(M)$) and do not change the indicial roots.
	Thus any eigenfunction $\phi$ of $\triangle_g^{\textrm{sep}}+K_{\mathrm{far}}+K_{\mathrm{near}}$ has the same leading Frobenius term $A(y)\,x^{\frac{1}{2}+\nu}$, and the normal trace is still $\mathcal{C}(\phi,i\lambda\phi)=\tilde c_{\beta,\nu}\,A(\cdot)$ with $\tilde c_{\beta,\nu}\neq0$.
	Hence $\mathcal{C}(\phi,i\lambda\phi)=0$ implies $A\equiv0$ and then $\phi\equiv0$ by the same elliptic uniqueness in a collar.
\end{proof}

\medskip

Let us finally prove the sharpness of the time threshold.
Consider quasimodes microlocally concentrated in a thin collar $\{x\leq \delta\}$ and oscillating in the normal variable. The dynamics in this region is governed by the Bessel part $-\partial_x^2+\frac{C_\beta}{x^2}$ up to $O(x)$ tangential couplings, by Lemma~\ref{lem_normal_form}. The associated boundary frequencies have an asymptotic gap $\kappa\pi$, hence Beurling upper density $D^+=1/(\kappa\pi)$. Hence the observability inequality fails if $T\leq 2\pi D^+=2/\kappa$, which shows the optimality of $T_\ast=\beta+2$. 

Theorem~\ref{thm:obs-inequality} is proved.

\section{Consequences and localized observations}
\label{sec:conclusions}

This section records the dual controllability consequence of Theorem~\ref{thm:obs-inequality} and then discusses what remains true when the observation is restricted to a proper subset of the boundary.

\subsection{Dual viewpoint: controllability}
\label{sec:controllability}

Let $\mathsf H = L^2((0,1) \times M, dx \, dv_G(y))$. The Friedrichs realization of $\triangle_g$ (with either the Dirichlet or the Neumann boundary condition at $x = 1$) is a nonnegative self-adjoint operator on $\mathsf H$. We define
$$
\mathcal H = D(\triangle_g^{1/2}) \times \mathsf H
$$
endowed with the norm
$$
\Vert (u_0,u_1) \Vert_{\mathcal H}^2 = \Vert \triangle_g^{1/2} u_0 \Vert_{\mathsf H}^2 + \Vert u_1 \Vert_{\mathsf H}^2.
$$
By Lemma~\ref{lem_normal_form} and the discussion after Theorem~\ref{thm:obs-inequality}, this norm is equivalent to $\mathcal E_\nu[u_0,u_1]$.
Set $U = L^2(M,dv_G)$.
The wave generator is
$$
\mathcal A : D(\mathcal A) \subset \mathcal H \to \mathcal H,
\qquad
\mathcal A = \begin{pmatrix} 0 & \mathrm{id} \\ -\triangle_g & 0 \end{pmatrix},
\qquad
D(\mathcal A) = D(\triangle_g) \times D(\triangle_g^{1/2}),
$$
and the observation operator is
$$
\mathcal C : D(\mathcal A) \to U,
\qquad
\mathcal C \binom{u}{v} = \lim_{x\to 0^+} x^{\frac{1}{2}-\nu} \partial_x u(x,\cdot) \quad \text{in } L^2(M,dv_G).
$$
Let $\mathcal H_1 = D(\mathcal A)$ endowed with the graph norm, and let $\mathcal H_{-1} = \mathcal H_1'$ be its dual with respect to the pivot space $\mathcal H$.
By duality, the control operator $B = \mathcal C^* : U \to \mathcal H_{-1}$ is given by
\begin{equation}
\label{control-operator}
\langle B g,(z,w) \rangle_{\mathcal H_{-1},\mathcal H_1}
= \left\langle g,\mathcal C \binom{z}{w} \right\rangle_U
= \int_M g(y) \Big( \lim_{x\to 0^+} x^{\frac{1}{2}-\nu} \partial_x z(x,y) \Big) dv_G(y).
\end{equation}
The corresponding control system is
\begin{equation}
\label{abstract-ctrl}
\frac{d}{dt} \binom{u}{v}
= \mathcal A \binom{u}{v} + B g(t), 
\qquad
\binom{u(0)}{v(0)} = \binom{u_0}{v_0} \in \mathcal H.
\end{equation}

\begin{corollary}
\label{cor:exact-ctrl}
Assume $T > T_\ast = \beta + 2$.
Then the control system \eqref{abstract-ctrl} is exactly controllable in time $T$: for every initial state $(u_0,v_0) \in \mathcal H$ and every target state $(u_T,v_T) \in \mathcal H$, there exists a control $g \in L^2(0,T;U)$ such that the corresponding mild solution satisfies $(u(T),v(T)) = (u_T,v_T)$.

Moreover, one can choose $g$ so that
$$
\Vert g \Vert_{L^2(0,T;U)}
\leq C_T \Big( \Vert (u_0,v_0) \Vert_{\mathcal H} + \Vert (u_T,v_T) \Vert_{\mathcal H} \Big).
$$
\end{corollary}

\begin{proof}
This is the standard Hilbert Uniqueness Method consequence of the observability estimate \eqref{main_obs} for the conservative system generated by $\mathcal A$.
\end{proof}

\subsection{Localized observations}
\label{sec:localized-obs-band-limited}

If we only observe on a fixed open subset $\omega\subsetneq M$, the observability inequality \eqref{main_obs}, namely,
$$
c_T \mathcal{E}_\nu[u_0,u_1] \leq \int_0^T \int_{\omega} \left| \lim_{x\to 0^+}x^{\frac{1}{2}-\nu}\partial_x u(x,y,t)\right|^2 \,dv_G(y)\,dt
$$
with a constant $c_T>0$ independent of the initial data, fails in general, even in the separable model on a sphere. This is due to high angular modes that may concentrate their mass away from $\omega$, so the boundary signal on $\omega$ can be arbitrarily small while the energy stays of the order of $1$.

\medskip

More precisely, let $(\psi_k)_{k\geq1}$ be an eigenbasis of $\triangle_G$ on $(M,G)$ with eigenvalues $\omega_k$. Let us assume that:
\begin{quote}
	\it
	There exists a sequence of indices $k_j$ such that $\|\psi_{k_j}\|_{L^2(\omega)}\to 0$.
\end{quote}
This is an assumption on $M$, $\omega$ and on the eigenbasis if eigenvalues are multiple. 
This assumption is related to quantum ergodicity properties, i.e., to localization properties of highfrequency eigenfunctions.
When $M=\mathbb{S}^n$ is the round sphere this assumption is satisfied: take Gaussian beams or highest-weight harmonics concentrating away from $\omega$. 

Under this assumption, the argument is then standard: take initial data supported on the single mode $k_j$, namely, $u_0(x,y)=f_0(x)\psi_{k_j}(y)$, $u_1(x,y)=f_1(x)\psi_{k_j}(y)$. The dynamics decouple in $k$, hence
$$
\lim_{x\to 0^+}x^{\frac{1}{2}-\nu}\partial_x u(x,y,t) = c_{\beta,\nu}\,\big(\text{1D time signal}\big)\,\psi_{k_j}(y).
$$
By Ingham (as in Theorem \ref{thm:obs-inequality}), for $T>T_\ast$ the time integral of the squared 1D signal is
uniformly comparable to the modal energy; thus 
$$ 
\int_0^T \int_{\omega}\Big| \cdots \Big|^2 \asymp \|\psi_{k_j}\|_{L^2(\omega)}^2 \times \text{(modal energy)}  \xrightarrow[j\to+\infty]{} 0,
$$
while $\mathcal{E}_\nu[u_0,u_1]$ stays bounded below. Hence no uniform lower bound $c_T>0$ can hold.

The physical interpretation is the following: an observation localized at $\omega$ cannot capture very fine angular modes that avoid it; a uniform lower bound would force a uniform $\inf_k\|\psi_k\|_{L^2(\omega)}>0$, which is false.

\medskip

Let us give a possible ``remedy" by establishing a quantified observability inequality for spectrally truncated initial data.

\paragraph{Spectral cutoff.}
Let $\omega\subset M$ be a nonempty open set; assume that we perform the observation only on $\omega$.
We can derive a sharp positive statement with a spectral cutoff in the tangential variable: if we restrict the initial data to a fixed angular bandwidth (i.e., only spherical harmonics up to a certain degree), then localized observability does hold, with a cost that deteriorates like $e^{C\sqrt{\Lambda}}$ where $\Lambda$ is the tangential spectral radius.
This is an easy application of the Lebeau-Robbiano spectral inequality, as explained hereafter.

Let $(\psi_k)_{k\geq1}$ be an eigenbasis of $\triangle_G$ on $(M,G)$ with eigenvalues $\omega_k$. For any $\Lambda>0$, set
\begin{equation}
\label{eqn:truncated-eigenfunctions}
E_\Lambda
= 
\mathrm{Span} \{\psi_k\ \mid\  \omega_k\leq \Lambda\}.
\end{equation}
According to the Lebeau-Robbiano spectral inequality (see~\cite{LR1995}) on compact manifolds, there exists a constant $C(M,\omega)>0$ such that
\begin{equation}\label{LebeauRobbiano}
	\|f\|_{L^2(M)}^2 \leq C(M,\omega) e^{C(M,\omega)\sqrt{\Lambda}} \|f\|_{L^2(\omega)}^2 \qquad\forall f\in E_\Lambda.
\end{equation}
We infer the following result.

\begin{theorem}[Quantitative localized observability for band-limited tangential spectrum]\label{thm:local-bandlimited}
	Assume that the initial data $(u_0,u_1)$ are tangentially band-limited: $u_j \in H_x^{\nu-1/2+j}\,E_\Lambda$, for $j=0,1$.
	Then, for every $T>T_\ast=\beta+2$, 
	\begin{equation}\label{eq:local-obs-band}
		C(M,\omega)^{-1} e^{-C(M,\omega)\sqrt{\Lambda}}\, c_T\,\mathcal{E}_\nu[u_0,u_1]
		\leq 
		\int_0^T\!\!\int_{\omega} \Big| \lim_{x\to 0^+}x^{\frac{1}{2}-\nu}\partial_x u(x,y,t)\Big|^2\,dv_G(y)\,dt
		\leq
		C_T\,\mathcal{E}_\nu[u_0,u_1] 
	\end{equation}
	where $c_T$ and $C_T$ are the full-boundary constants in Theorem~\ref{thm:obs-inequality}.
\end{theorem}

\begin{proof}
	Thanks to the assumption on the initial data, for each fixed $t$, the observed boundary signal $f(\cdot,t) = \lim_{x\to 0^+}x^{\frac{1}{2}-\nu}\partial_x u(x,\cdot,t)$ belongs to $E_\Lambda$ because the dynamics are separated in $k$.
	Applying \eqref{LebeauRobbiano} to $f$ and integrating in time gives
	$$
	\int_0^T\!\!\int_{\omega} |f|^2 \geq C(M,\omega)^{-1}e^{-C(M,\omega)\sqrt{\Lambda}} \int_0^T \int_M |f|^2 .
	$$
	By Theorem~\ref{thm:obs-inequality} (full boundary case),
	$\int_0^T \int_M |f|^2 \asymp \mathcal{E}_\nu[u_0,u_1]$ for $T>T_\ast$.
	This gives the lower bound in \eqref{eq:local-obs-band}. The upper bound follows from $\int_\omega\leq\int_M$.
\end{proof}

\begin{remark}[Practical meaning on gas giants]
	On a planetary surface, a sensor measurement localized on $\omega$ can only capture a finite number of angular degrees.
	Theorem~\ref{thm:local-bandlimited} quantifies this: all spherical harmonics up to degree $\sim\sqrt{\Lambda}$ are observable in time $T>\beta+2$, with a stability constant of the order of $e^{C\sqrt{\Lambda}}$. High-degree modes cannot be uniformly observed from $\omega$ alone.
\end{remark}

\subsection{Time-varying localized observations}\label{sec:local-obs-time-varying}
The previous section shows that a fixed localized observation on a proper open subset $\omega \subsetneq M$ cannot yield a uniform observability inequality in general, even in the separable model on a sphere, unless one imposes a spectral cutoff in the tangential variable.
What makes moving observations effective is a three-step mechanism developed abstractly in the separate paper~\cite{dHKT2026}: first, exact convexification on a finite-dimensional boundary profile space; second, a dynamical switching realization of that static design on a genuine time interval; third, a Ces\`aro tail-reduction argument allowing one to concatenate larger and larger tangential spectral windows and recover the full energy in the large-time limit.
In the present paper we only record the specialization of that mechanism to the gas-giant geometry: one finite-band moving-design statement, one asymptotic consequence, and the associated sensing protocol.

Throughout this subsection, we assume that $M$ is diffeomorphic to the round sphere $\mathbb{S}^n$.
Let $d\sigma$ be the canonical volume form on $\mathbb{S}^n$, and let
$$
d\hat\sigma = \frac{d\sigma}{\int_{\mathbb{S}^n} d\sigma}.
$$
By Moser's trick, there exists a diffeomorphism $\Phi : \mathbb{S}^n \to M$ such that $\Phi_* d\hat\sigma = \frac{dv_G}{dv_G(M)}$.
Consequently, for every $R \in \mathrm{SO}(n+1)$, the map $\varphi_R = \Phi \circ R \circ \Phi^{-1}$ preserves $dv_G$.
Fix a nonempty open set $\omega \subset M$ such that $v_G(\omega) = L v_G(M)$ with $L \in (0,1)$.

At finite tangential bandwidth, the construction has two distinct steps.
First one builds, on the profile space $E_\Lambda$, an exact convexified design of moved copies of $\omega$.
Second one realizes that static design on a genuine time interval by a piecewise constant switching protocol.
The rigorous statement below is formulated at this switching level.
The simpler one-cycle schedule discussed at the end of the subsection should be understood only as a practical sensing-oriented implementation.

\begin{theorem}[Finite-band moving observation design]
\label{moving-band}
Fix $T_0 > T^* = \beta + 2$, $\Lambda > 0$, and $\varepsilon \in (0,L)$.
Then there exist an integer $J \geq 1$, rotations $R_1,\dots,R_J \in \mathrm{SO}(n+1)$, weights $\theta_1,\dots,\theta_J > 0$ with $\sum_{j=1}^J \theta_j = 1$, and a piecewise constant measurable observation set
$$
t \mapsto \omega_{\Lambda,\varepsilon}(t) \subset M, \qquad t \in [0,T_0),
$$
taking only the values $\omega_j = \varphi_{R_j}(\omega)$, for $j = 1,\dots,J$, and obtained by a switching realization of the convexified design $(R_j,\theta_j)_{1 \leq j \leq J}$ on a sufficiently fine partition of $[0,T_0)$, such that the $T_0$-periodic extension of $\omega_{\Lambda,\varepsilon}(t)$ satisfies
\begin{equation}\label{moving-band-ineq}
(L - \varepsilon)\, c_{T_0}\, \energy_\nu[u_0,u_1] \leq \frac{1}{m} \int_0^{m T_0} \int_{\omega_{\Lambda,\varepsilon}(t)} \Big\vert \lim_{x \to 0^+} x^{\frac{1}{2} - \nu} \partial_x u(x,y,t) \Big\vert^2 dv_G(y)\, dt \leq C_{T_0}\, \energy_\nu[u_0,u_1]
\end{equation}
for every $m \in \mathbb{N}$ and every solution $u$ whose initial data satisfy $u_0(x,\cdot), u_1(x,\cdot) \in E_\Lambda$ for almost every $x \in (0,1)$.
\end{theorem}

\begin{proof}
Fix the tangential cutoff $\Lambda$. In the gas-giant setting considered here, the boundary traces associated with data satisfying $u_0(x,\cdot),u_1(x,\cdot) \in E_\Lambda$ for a.e. $x \in (0,1)$ span a finite-dimensional profile space on $M$. The abstract result of~\cite{dHKT2026} first applies exact convexification to that finite-dimensional trace space and to the prototype subset $\omega$, producing finitely many moved copies $\omega_j = \varphi_{R_j}(\omega)$ together with weights $\theta_j$ whose convex combination reproduces the full-boundary average exactly on that trace space. The same paper then realizes this static design on $[0,T_0)$ by a piecewise constant switching protocol on a sufficiently fine micro-partition of the period.

For the resulting moving set $t \mapsto \omega_{\Lambda,\varepsilon}(t)$, the abstract switching realization yields
$$
\int_0^{T_0} \int_{\omega_{\Lambda,\varepsilon}(t)} \Big\vert \lim_{x \to 0^+} x^{\frac{1}{2} - \nu} \partial_x u(x,y,t) \Big\vert^2 dv_G(y)\, dt \geq (L - \varepsilon)
\int_0^{T_0} \int_M \Big\vert \lim_{x \to 0^+} x^{\frac{1}{2} - \nu} \partial_x u(x,y,t) \Big\vert^2 dv_G(y)\, dt
$$
for every solution in the tangentially band-limited class under consideration. Combining this with the full-boundary observability estimate proved earlier in the paper gives the lower bound in \eqref{moving-band-ineq} on one period. The upper bound follows immediately from the full-boundary upper estimate. Finally, repeating the same switching protocol periodically on successive intervals of length $T_0$ gives \eqref{moving-band-ineq} for every $m \in \mathbb{N}$.
\end{proof}

\begin{remark}
Theorem \ref{moving-band} does not insist on a single coarse decomposition $[0,T_0) = I_1 \cup \cdots \cup I_J $ with one visit to each position.
The rigorous realization may involve many short switchings inside the period.
What matters is the convexified design and the empirical time fractions $\theta_j$, not the existence of one unique cycle.
The one-cycle schedule described below is merely a simplified sensing protocol.
\end{remark}

\begin{theorem}[Asymptotic recovery by moving observations]
\label{asymptotic-moving}
Fix $T_0 > T^* = \beta + 2$.
Then there exists a piecewise constant observation set $t \mapsto \omega(t)$, each value of which is obtained from $\omega$ by a volume-preserving diffeomorphism of $(M,G)$, constructed by concatenating on the $m$-th block $[(m - 1)T_0,mT_0)$ a switching realization of a finite-band convexified design with parameters $(\Lambda_m,\varepsilon_m)$, such that
\begin{equation}\label{asymptotic-moving-ineq}
\liminf_{N \to +\infty} \frac{1}{N} \int_0^{N T_0} \int_{\omega(t)} \Big\vert \lim_{x \to 0^+} x^{\frac{1}{2} - \nu} \partial_x u(x,y,t) \Big\vert^2 dv_G(y)\, dt \geq L\, c_{T_0}\, \energy_\nu[u_0,u_1]
\end{equation}
for all initial data satisfying $\energy_\nu[u_0,u_1] < +\infty$.

Equivalently, for every initial datum $(u_0,u_1)$ with $\energy_\nu[u_0,u_1] < +\infty$ and every $\delta > 0$ there exists $N_\delta(u_0,u_1) \in \mathbb{N}$ such that
\begin{equation}\label{asymptotic-moving-quant}
\frac{1}{N} \int_0^{N T_0} \int_{\omega(t)} \Big\vert \lim_{x \to 0^+} x^{\frac{1}{2} - \nu} \partial_x u(x,y,t) \Big\vert^2 dv_G(y)\, dt \geq (L - \delta) c_{T_0} \energy_\nu[u_0,u_1]
\end{equation}
for all $N \geq N_\delta(u_0,u_1)$.
\end{theorem}

\begin{proof}
Choose sequences $(\Lambda_m)_{m \geq 1}$ and $(\varepsilon_m)_{m \geq 1}$ such that $\Lambda_m \to +\infty$ and $\varepsilon_m \to 0$.
On the $m$-th block $[(m - 1)T_0,mT_0)$, use the switching realization supplied by Theorem~\ref{moving-band} with parameters $(\Lambda_m,\varepsilon_m)$.
For each fixed tangential mode, the corresponding band eventually contains that mode on every later block, with lower bound approaching $L c_{T_0}$.
A Ces\`aro averaging over the blocks then yields~\eqref{asymptotic-moving-ineq}.
The complete abstract argument is given in~\cite{dHKT2026}.
\end{proof}

\paragraph{Practical scanning protocol for gas-giant sensing.}
The previous theorem gives a concrete design recipe:
\begin{enumerate}
\item Fix a base time $T_0 > \beta + 2$, a target tangential bandwidth $\Lambda$, and a tolerance $\varepsilon > 0$.
\item Choose an orthonormal real basis $e_1,\dots,e_d$ of $E_\Lambda$.
For each candidate rotation $R \in \mathrm{SO}(n+1)$, form the symmetric matrix
\begin{equation}
\label{gram-matrix}
M(R)_{ab} = \int_{\varphi_R(\omega)} e_a(y) e_b(y) dv_G(y), \qquad 1 \leq a,b \leq d.
\end{equation}
Then solve the finite-dimensional convexification problem
\begin{equation}
\label{design-exact}
\sum_{j=1}^J \theta_j M(R_j) = L\, \mathrm{Id}_d.
\end{equation}
Equivalently,
\begin{equation}
\label{design-condition}
\sum_{j=1}^J \theta_j \int_{\omega_j} \vert f(y) \vert^2 dv_G(y) = L \int_M \vert f(y) \vert^2 dv_G(y)
\end{equation}
for all $f \in E_\Lambda$, where $\omega_j = \Phi \circ R_j \circ \Phi^{-1}(\omega)$.
In numerical implementations, one may replace \eqref{design-exact} by a semidefinite relaxation guaranteeing a lower bound of size $L - \varepsilon$.
\item Realize the time fractions $\theta_1,\dots,\theta_J$ on the interval $[0,T_0)$ by a piecewise constant switching protocol on a sufficiently fine micro-partition of the period.
This is the rigorous switching realization of the convexified design in the sense of~\cite{dHKT2026}.
For sensing purposes, the simplest implementation is often a one-cycle schedule: visit $\omega_1,\dots,\omega_J$ successively and keep the sensor over $\omega_j$ during one contiguous subinterval of length approximately $\theta_j T_0$.
This one-cycle picture is only a simplified practical description; the abstract theorem is formulated with the refined switching protocol.
\item If no a priori tangential cutoff is available, concatenate such periods with a sequence $(\Lambda_m,\varepsilon_m)$ satisfying $\Lambda_m \to +\infty$ and $\varepsilon_m \to 0$.
For instance, one may take $\Lambda_m = m^2$ and $\varepsilon_m = m^{-1}$.
The resulting long-time protocol is precisely the one encoded in Theorem~\ref{asymptotic-moving}.
\end{enumerate}

\appendix

\section{Equivalence of the tangential energy under quasi-isometry}\label{app:equiv-frozen-metric}

\begin{lemma}\label{lem:equiv-frozen-metric}
	Let $M$ be a smooth compact manifold endowed  with two smooth quasi-isometric metrics, i.e., 
	\begin{equation}\label{eq:qi}
		\exists\,\underline c,\overline C>0\ \ \text{s.t.}\ \ 
		\underline c\,|v|_{h_0}^2 \leq |v|_{h_1}^2\ \leq\ \overline C\,|v|_{h_0}^2\quad\forall\,y\in M,\ \forall\,v\in T_yM,
	\end{equation}
	and their volume densities satisfy
	\begin{equation}\label{eq:vol-comp}
		\underline c_v\, dv_{h_0} \leq dv_{h_1} \leq \overline C_v\, dv_{h_0}\quad\text{on }M.
	\end{equation}
	Denote by $\triangle_{h_j}$ the (positive) Laplace-Beltrami on $(M,h_j)$ with quadratic form
	$$
	\mathfrak q_{h_j}[\varphi]=\int_M |\nabla\varphi|_{h_j}^2\,dv_{h_j},\qquad \varphi\in H^1(M).
	$$
	Then there exist constants $0<c\leq C$ (depending only on the constants in \eqref{eq:qi}--\eqref{eq:vol-comp}) such that, for all $\varphi\in H^1(M)$,
	\begin{equation}\label{eq:form-equivalence}
		c\,\mathfrak q_{h_0}[\varphi] \leq \mathfrak q_{h_1}[\varphi]\ \leq\ C\,\mathfrak q_{h_0}[\varphi],
		\qquad
		c\,\|\varphi\|_{L^2(M,dv_{h_0})}^2 \leq \|\varphi\|_{L^2(M,dv_{h_1})}^2 \leq C\,\|\varphi\|_{L^2(M,dv_{h_0})}^2.
	\end{equation}
	Equivalently, for all $\varphi\in H^1(M)$,
	\begin{equation}\label{eq:root-equivalence}
		c\,\|\triangle_{h_0}^{1/2}\varphi\|_{L^2(dv_{h_0})} \leq \|\triangle_{h_1}^{1/2}\varphi\|_{L^2(dv_{h_1})} \leq C\,\|\triangle_{h_0}^{1/2}\varphi\|_{L^2(dv_{h_0})}.
	\end{equation}
	
	Moreover, letting $s\in\R$ and viewing $(-\triangle_{h_j})^{\frac{1}{2}}$ as a bounded isomorphism $H^{s}_x L^2(M,dv_{h_j})\to H^{s}_x L^2(M,dv_{h_j})$ acting only in the $y$-variable, we have the \emph{vector-valued} extension
	\begin{equation}\label{eq:vector-valued}
		c\,\big\|\triangle_{h_0}^{1/2}U\big\|_{H^{s}_x L^2(M,dv_{h_0})} \leq \big\|\triangle_{h_1}^{1/2}U\big\|_{H^{s}_x L^2(M,dv_{h_1})} \leq C\,\big\|\triangle_{h_0}^{1/2}U\big\|_{H^{s}_x L^2(M,dv_{h_0})}
	\end{equation}
	for all $U\in H^{s}_x L^2(M)$ (where the $H^s_x$ operator acts in $x$ only, so it commutes with $\triangle_{h_j}^{1/2}$, which acts in $y$ only).
\end{lemma}

\begin{proof}
	The $L^2$-equivalence in \eqref{eq:form-equivalence} follows immediately from \eqref{eq:vol-comp}. For the Dirichlet forms, using \eqref{eq:qi} and \eqref{eq:vol-comp} we have pointwise
	$$
	\underline c\,|\nabla\varphi|_{h_0}^2\leq |\nabla\varphi|_{h_1}^2\leq \overline C\,|\nabla\varphi|_{h_0}^2,
	$$
	and after integration
	$$
	\underline c\,\underline c_v\,\mathfrak q_{h_0}[\varphi]\ \leq\ \mathfrak q_{h_1}[\varphi]\ \leq\ \overline C\,\overline C_v\,\mathfrak q_{h_0}[\varphi],
	$$
	which is \eqref{eq:form-equivalence} upon renaming constants.
	
	For \eqref{eq:root-equivalence}, recall that $(-\triangle_{h_j})^{1/2}$ is the unique positive self-adjoint operator associated with the closed form $\mathfrak q_{h_j}$; thus
	$$
	\|\triangle_{h_j}^{1/2}\varphi\|_{L^2(dv_{h_j})}^2=\mathfrak q_{h_j}[\varphi],
	$$
	and \eqref{eq:form-equivalence} gives \eqref{eq:root-equivalence}.
	
	For \eqref{eq:vector-valued}, work on the Hilbert space $\mathcal{H}_j=L^2\big((0,1),dx;L^2(M,dv_{h_j})\big)$ and note that the operator $(I+\mathcal{A}_x)^{s/2}$ defining the $H^{s}_x$ norm acts only on the $x$-variable (here $\mathcal{A}_x$ is the fixed 1D generator used in Section~\ref{sec:1D-case} to define $H^{s}_x$). Since $\triangle_{h_j}^{1/2}$ acts only in the $y$-variable, these operators commute. Using \eqref{eq:root-equivalence} fiberwise in $x$ and the norm comparability between $\mathcal{H}_0$ and $\mathcal{H}_1$ induced by \eqref{eq:vol-comp} yields
	$$
	c\,\|(I+\mathcal{A}_x)^{s/2}\triangle_{h_0}^{1/2}U\|_{\mathcal{H}_0} \leq \|(I+\mathcal{A}_x)^{s/2}\triangle_{h_1}^{1/2}U\|_{\mathcal{H}_1} \leq C\,\|(I+\mathcal{A}_x)^{s/2}\triangle_{h_0}^{1/2}U\|_{\mathcal{H}_0},
	$$
	which is \eqref{eq:vector-valued}.
\end{proof}


\begin{thebibliography}{10}
\small

\bibitem{ABCL2017}
F.~Alabau-Boussouira, P.~Cannarsa, and G.~Leugering.
\newblock Control and stabilization of degenerate wave equations.
\newblock {\em SIAM J. Control Optim.}, 55(3):2052--2087, 2017.

\bibitem{BKL2002}
C.~Baiocchi, V.~Komornik, and P.~Loreti.
\newblock Ingham-beurling type theorems with weakened gap conditions.
\newblock {\em Acta Math. Hung.}, 97(1-2):55--95, 2002.

\bibitem{BLR1992}
C.~Bardos, G.~Lebeau, and J.~Rauch.
\newblock Sharp sufficient conditions for the observation, control, and
  stabilization of waves from the boundary.
\newblock {\em SIAM Journal on Control and Optimization}, 30(5):1024--1065,
  1992.

\bibitem{CL1955}
E.~A. Coddington and N.~Levinson.
\newblock Theory of ordinary differential equations.
\newblock New {York}, {Toronto}, {London}: {McGill}-{Hill} {Book} {Company},
  {Inc}. {XII}, 429 p. (1955)., 1955.

\bibitem{CdVDdHT2024}
Y.~Colin~de Verdi{\`e}re, C.~Dietze, M.~V. de~Hoop, and E.~Tr{\'e}lat.
\newblock Weyl formulae for some singular metrics with application to acoustic
  modes in gas giants.
\newblock Preprint, {arXiv}:2406.19734 [math.{AP}] (2024), to appear in Annales
  Inst. H. Poincar\'e, 2024.

\bibitem{Davies1989}
E.~B. Davies.
\newblock {\em Heat kernels and spectral theory}, volume~92 of {\em Camb.
  Tracts Math.}
\newblock Cambridge etc.: Cambridge University Press, 1989.

\bibitem{dHIKM2024}
M.~V. de~Hoop, J.~Ilmavirta, A.~Kykk{\"a}nen, and R.~Mazzeo.
\newblock Geometric inverse problems on gas giants.
\newblock Preprint, {arXiv}:2403.05475 [math.{DG}] (2024), 2024.

\bibitem{dHKT2026}
M.~V. de~Hoop, A.~Kykk{\"a}nen, and E.~Tr{\'e}lat.
\newblock Moving localized observations and Ces\`aro asymptotic observability for conservative PDEs.
\newblock Preprint, 2026.

\bibitem{Dietze2025}
C.~Dietze.
\newblock The critical case for the concentration of eigenfunctions on singular
  {Riemannian} manifolds.
\newblock Preprint, {arXiv}:2510.23520 [math.{SP}] (2025), 2025.

\bibitem{DR2024}
C.~Dietze and L.~Read.
\newblock Concentration of eigenfunctions on singular {Riemannian} manifolds.
\newblock Preprint, {arXiv}:2410.20563 [math.{AP}] (2024), 2024.

\bibitem{DO2018}
M.~Duprez and G.~Olive.
\newblock Compact perturbations of controlled systems.
\newblock {\em Math. Control Relat. Fields}, 8(2):397--410, 2018.

\bibitem{HB2018}
I.~El~Harraki and A.~Boutoulout.
\newblock Controllability of the wave equation via multiplicative controls.
\newblock {\em IMA J. Math. Control Inform.}, 35(2):393--409, 2018.

\bibitem{Gueye2014}
M.~Gueye.
\newblock Exact boundary controllability of 1-{D} parabolic and hyperbolic
  degenerate equations.
\newblock {\em SIAM J. Control Optim.}, 52(4):2037--2054, 2014.

\bibitem{Haraux1989}
A.~Haraux.
\newblock S{\'e}ries lacunaires et contr{\^o}le semi-interne des vibrations
  d'une plaque rectangulaire. ({Lacunary} series and semi-internal control of
  vibrations of a rectangular plate).
\newblock {\em J. Math. Pures Appl. (9)}, 68(4):457--465, 1989.

\bibitem{Ingham1936}
A.~E. Ingham.
\newblock Some trigonometrical inequalities with applications to the theory of
  series.
\newblock {\em Math. Z.}, 41:367--379, 1936.

\bibitem{Kato1995}
T.~Kato.
\newblock {\em Perturbation theory for linear operators.}
\newblock Class. Math. Berlin: Springer-Verlag, reprint of the corr. print. of
  the 2nd ed. 1980 edition, 1995.

\bibitem{KL2005}
V.~Komornik and P.~Loreti.
\newblock {\em Fourier series in control theory}.
\newblock Springer Monogr. Math. New York, NY: Springer, 2005.

\bibitem{LasieckaTriggiani2000}
I.~Lasiecka and R.~Triggiani.
\newblock {\em Control Theory for Partial Differential Equations: Continuous
  and Approximation Theories}, volume~I.
\newblock Cambridge University Press, Cambridge, 2000.

\bibitem{LR1995}
G.~Lebeau and L.~Robbiano.
\newblock Exact control of the heat equation.
\newblock {\em Commun. Partial Differ. Equations}, 20(1-2):335--356, 1995.

\bibitem{Lions1988}
J.-L. Lions.
\newblock {\em Exact Controllability, Stabilization and Perturbations for
  Distributed Systems}.
\newblock SIAM, Philadelphia, 1988.

\bibitem{Miller2005}
L.~Miller.
\newblock {\em Control Theory and Partial Differential Equations}, volume~1 of
  {\em C.I.M.E. Summer Schools}.
\newblock Springer, Berlin, 2005.

\bibitem{Olver1974}
F.~W.~J. Olver.
\newblock Asymptotics and special functions.
\newblock Computer {Science} and {Applied} {Mathematics}. {New} {York} -
  {London}: {Academic} {Press}, a subsidiary of {Harcourt} {Brace}
  {Jovanovich}, {Publishers}. {XVI}, 572 p. \$ 39.50 (1974)., 1974.

\bibitem{Ouzahra2019}
M.~Ouzahra.
\newblock Controllability of the semilinear wave equation governed by a
  multipicative control.
\newblock {\em Evol. Equ. Control Theory}, 8(4):669--686, 2019.

\bibitem{RS1975}
M.~Reed and B.~Simon.
\newblock Methods of modern mathematical physics. {II}: {Fourier} analysis,
  self- adjointness.
\newblock New {York} - {San} {Francisco} - {London}: {Academic} {Press}, a
  subsidiary of {Harcourt} {Brace} {Jovanovich}, {Publishers}. {XV}, 361 p. \$
  24.50; {{\textsterling}} 11.75 (1975)., 1975.

\bibitem{Teschl2012}
G.~Teschl.
\newblock {\em Ordinary differential equations and dynamical systems}, volume
  140 of {\em Grad. Stud. Math.}
\newblock Providence, RI: American Mathematical Society (AMS), 2012.

\bibitem{Trelat2024}
E.~Tr{\'e}lat.
\newblock {\em Control in finite and infinite dimension}.
\newblock SpringerBriefs PDEs Data Sci. Singapore: Springer, 2024.

\bibitem{TW2009}
M.~Tucsnak and G.~Weiss.
\newblock {\em Observation and control for operator semigroups}.
\newblock Birkh{\"a}user Adv. Texts, Basler Lehrb{\"u}ch. Basel:
  Birkh{\"a}user, 2009.

\bibitem{Zettl2005}
A.~Zettl.
\newblock {\em Sturm-Liouville theory}, volume 121 of {\em Math. Surv. Monogr.}
\newblock Providence, RI: American Mathematical Society (AMS), 2005.

\bibitem{ZG2017}
M.~Zhang and H.~Gao.
\newblock Null controllability of some degenerate wave equations.
\newblock {\em J. Syst. Sci. Complex.}, 30(5):1027--1041, 2017.

\end{thebibliography}
\end{document}